\documentclass[11pt]{amsart}
\usepackage{amsmath,amsfonts,mathrsfs}

\newif\ifdraft
\draftfalse
\usepackage{amsmath,amsfonts,amsthm,mathrsfs}
\usepackage{amssymb}

\usepackage[unicode,bookmarks]{hyperref}
\usepackage[usenames,dvipsnames]{xcolor}
\hypersetup{colorlinks=true,citecolor=NavyBlue,linkcolor=Brown,urlcolor=Orange}

\usepackage[alphabetic,initials]{amsrefs}

\usepackage{enumitem}

\usepackage{chngcntr}

\ifdraft
\usepackage[notcite,notref,color]{showkeys}

\definecolor{labelkey}{gray}{0.5}
\fi

\usepackage{tikz}

\usepackage{tikz-cd}

\usetikzlibrary{matrix,arrows}
\newlength{\myarrowsize} 

\pgfarrowsdeclare{cmto}{cmto}{
	\pgfsetdash{}{0pt} 
	\pgfsetbeveljoin 
	\pgfsetroundcap 
	\setlength{\myarrowsize}{0.6pt}
	\addtolength{\myarrowsize}{.5\pgflinewidth}
	\pgfarrowsleftextend{-4\myarrowsize-.5\pgflinewidth} 
	\pgfarrowsrightextend{.8\pgflinewidth}
}{
	\setlength{\myarrowsize}{0.6pt} 
  	\addtolength{\myarrowsize}{.5\pgflinewidth}  
	\pgfsetlinewidth{0.5\pgflinewidth}
	\pgfsetroundjoin
	\pgfpathmoveto{\pgfpoint{1.5\pgflinewidth}{0}}
	\pgfpatharc{-109}{-170}{4\myarrowsize}
	\pgfpatharc{10}{189}{0.58\pgflinewidth and 0.2\pgflinewidth}
	\pgfpatharc{-170}{-115}{4\myarrowsize+\pgflinewidth}
	\pgfpathclose
	\pgfusepathqfillstroke
	\pgfpathmoveto{\pgfpoint{1.5\pgflinewidth}{0}}
	\pgfpatharc{109}{170}{4\myarrowsize}
	\pgfpatharc{-10}{-189}{0.58\pgflinewidth and 0.2\pgflinewidth}
	\pgfpatharc{170}{115}{4\myarrowsize+\pgflinewidth}
	\pgfpathclose
	\pgfusepathqfillstroke
	\pgfsetlinewidth{2\pgflinewidth}
}

\pgfarrowsdeclare{cmonto}{cmonto}{
	\pgfsetdash{}{0pt} 
	\pgfsetbeveljoin 
	\pgfsetroundcap 
	\setlength{\myarrowsize}{0.6pt}
	\addtolength{\myarrowsize}{.5\pgflinewidth}
	\pgfarrowsleftextend{-4\myarrowsize-.5\pgflinewidth} 
	\pgfarrowsrightextend{.8\pgflinewidth}
}{
	\setlength{\myarrowsize}{0.6pt} 
  	\addtolength{\myarrowsize}{.5\pgflinewidth}  
	\pgfsetlinewidth{0.5\pgflinewidth}
	\pgfsetroundjoin
	\pgfpathmoveto{\pgfpoint{1.5\pgflinewidth}{0}}
	\pgfpatharc{-109}{-170}{4\myarrowsize}
	\pgfpatharc{10}{189}{0.58\pgflinewidth and 0.2\pgflinewidth}
	\pgfpatharc{-170}{-115}{4\myarrowsize+\pgflinewidth}
	\pgfpathclose
	\pgfusepathqfillstroke
	\pgfpathmoveto{\pgfpoint{1.5\pgflinewidth}{0}}
	\pgfpatharc{109}{170}{4\myarrowsize}
	\pgfpatharc{-10}{-189}{0.58\pgflinewidth and 0.2\pgflinewidth}
	\pgfpatharc{170}{115}{4\myarrowsize+\pgflinewidth}
	\pgfpathclose
	\pgfusepathqfillstroke
	\pgfpathmoveto{\pgfpoint{1.5\pgflinewidth-0.3em}{0}}
	\pgfpatharc{-109}{-170}{4\myarrowsize}
	\pgfpatharc{10}{189}{0.58\pgflinewidth and 0.2\pgflinewidth}
	\pgfpatharc{-170}{-115}{4\myarrowsize+\pgflinewidth}
	\pgfpathclose
	\pgfusepathqfillstroke
	\pgfpathmoveto{\pgfpoint{1.5\pgflinewidth-0.3em}{0}}
	\pgfpatharc{109}{170}{4\myarrowsize}
	\pgfpatharc{-10}{-189}{0.58\pgflinewidth and 0.2\pgflinewidth}
	\pgfpatharc{170}{115}{4\myarrowsize+\pgflinewidth}
	\pgfpathclose
	\pgfusepathqfillstroke
	\pgfsetlinewidth{2\pgflinewidth}
}

\pgfarrowsdeclare{cmhook}{cmhook}{
	\pgfsetdash{}{0pt} 
	\pgfsetbeveljoin 
	\pgfsetroundcap 
	\setlength{\myarrowsize}{0.6pt}
	\addtolength{\myarrowsize}{.5\pgflinewidth}
	\pgfarrowsleftextend{-4\myarrowsize-.5\pgflinewidth} 
	\pgfarrowsrightextend{.8\pgflinewidth}
}{
	\setlength{\myarrowsize}{0.6pt} 
  	\addtolength{\myarrowsize}{.5\pgflinewidth}  
 	\pgfsetdash{}{0pt}
	\pgfsetroundcap
	\pgfpathmoveto{\pgfqpoint{0pt}{-4.667\pgflinewidth}}
	\pgfpathcurveto
    {\pgfqpoint{4\pgflinewidth}{-4.667\pgflinewidth}}
    {\pgfqpoint{4\pgflinewidth}{0pt}}
    {\pgfpointorigin}
	\pgfusepathqstroke
}

\newenvironment{diagram*}[2]{%
\[%
\begin{tikzpicture}[>=cmto,baseline=(current bounding box.center),%
	to/.style={->,font=\scriptsize,cap=round},%
	into/.style={cmhook->,font=\scriptsize,cap=round},%
	onto/.style={-cmonto,font=\scriptsize,cap=round},%
	math/.style={matrix of math nodes, row sep=#2, column sep=#1,%
		text height=1.5ex, text depth=0.25ex}]%
}{%
\end{tikzpicture}%
\]%
\ignorespacesafterend%
}

\newcommand{\Dmod}{\mathscr{D}}

\newcommand{\Mmod}{\mathcal{M}}

\newcommand{\derR}{\mathbf{R}}

\newcommand{\shH}{\mathcal{H}}

\newcommand{\ZZ}{\mathbb{Z}}
\newcommand{\QQ}{\mathbb{Q}}

\newcommand{\CC}{\mathbb{C}}

\newcommand*{\sheafhom}{\mathscr{H}\kern -.5pt om}

\newcommand{\shf}[1]{\mathscr{#1}}

\def\overbar#1#2#3{{%
	\setbox0=\hbox{\displaystyle{#1}}%
	\dimen0=\wd0
	\advance\dimen0 by -#2 
	\vbox {\nointerlineskip \moveright #3 \vbox{\hrule height 0.3pt width \dimen0}%
		\nointerlineskip \vskip 1.5pt \box0}%
}}

\newcommand{\shF}{\shf{F}}

\newcommand{\shE}{\shf{E}}

\newcommand{\shO}{\shf{O}}

\makeatletter
\let\@@seccntformat\@seccntformat
\renewcommand*{\@seccntformat}[1]{%
  \expandafter\ifx\csname @seccntformat@#1\endcsname\relax
    \expandafter\@@seccntformat
  \else
    \expandafter
      \csname @seccntformat@#1\expandafter\endcsname
  \fi
    {#1}%
}
\newcommand*{\@seccntformat@subsection}[1]{%
  \textbf{\csname the#1\endcsname.}
}
\makeatother

\makeatletter
\let\@paragraph\paragraph
\renewcommand*{\paragraph}[1]{%
	\vspace{0.3\baselineskip}%
	\@paragraph{\textit{#1}}%
}
\makeatother

\counterwithin{equation}{subsection}
\counterwithout{subsection}{section}
\counterwithin{figure}{subsection}

\newtheorem{theorem}[equation]{Theorem}
\newtheorem*{theorem*}{Theorem}
\newtheorem{lemma}[equation]{Lemma}
\newtheorem*{lemma*}{Lemma}
\newtheorem{corollary}[equation]{Corollary}
\newtheorem{proposition}[equation]{Proposition}
\newtheorem*{proposition*}{Proposition}

\theoremstyle{definition}
\newtheorem{definition}[equation]{Definition}
\newtheorem*{definition*}{Definition}
\newtheorem{remark}[equation]{Remark}

\newtheorem*{example*}{Example}
\newtheorem*{problem*}{Problem}

\theoremstyle{plain}

\newcommand{\theoremref}[1]{\hyperref[#1]{Theorem~\ref*{#1}}}
\newcommand{\lemmaref}[1]{\hyperref[#1]{Lemma~\ref*{#1}}}
\newcommand{\definitionref}[1]{\hyperref[#1]{Definition~\ref*{#1}}}
\newcommand{\propositionref}[1]{\hyperref[#1]{Proposition~\ref*{#1}}}
\newcommand{\conjectureref}[1]{\hyperref[#1]{Conjecture~\ref*{#1}}}
\newcommand{\corollaryref}[1]{\hyperref[#1]{Corollary~\ref*{#1}}}
\newcommand{\exampleref}[1]{\hyperref[#1]{Example~\ref*{#1}}}

\makeatletter
\let\old@caption\caption
\renewcommand*{\caption}[1]{%
	\setcounter{figure}{\value{equation}}%
	\stepcounter{equation}%
	\old@caption{#1}\relax%
}
\makeatother

\newcounter{intro}

\newtheorem{intro-conjecture}[intro]{Conjecture}
\newtheorem{intro-corollary}[intro]{Corollary}
\newtheorem{intro-proposition}[intro]{Proposition}
\newtheorem{intro-theorem}[intro]{Theorem}

\newcommand{\parref}[1]{\hyperref[#1]{\S\ref*{#1}}}

\makeatletter
\newcommand*\if@single[3]{%
  \setbox0\hbox{${\mathaccent"0362{#1}}^H$}%
  \setbox2\hbox{${\mathaccent"0362{\kern0pt#1}}^H$}%
  \ifdim\ht0=\ht2 #3\else #2\fi
  }
\newcommand*\rel@kern[1]{\kern#1\dimexpr\macc@kerna}
\newcommand*\widebar[1]{\@ifnextchar^{{\wide@bar{#1}{0}}}{\wide@bar{#1}{1}}}
\newcommand*\wide@bar[2]{\if@single{#1}{\wide@bar@{#1}{#2}{1}}{\wide@bar@{#1}{#2}{2}}}
\newcommand*\wide@bar@[3]{%
  \begingroup
  \def\mathaccent##1##2{%
    \if#32 \let\macc@nucleus\first@char \fi
    \setbox\z@\hbox{$\macc@style{\macc@nucleus}_{}$}%
    \setbox\tw@\hbox{$\macc@style{\macc@nucleus}{}_{}$}%
    \dimen@\wd\tw@
    \advance\dimen@-\wd\z@
    \divide\dimen@ 3
    \@tempdima\wd\tw@
    \advance\@tempdima-\scriptspace
    \divide\@tempdima 10
    \advance\dimen@-\@tempdima
    \ifdim\dimen@>\z@ \dimen@0pt\fi
    \rel@kern{0.6}\kern-\dimen@
    \if#31
      \overline{\rel@kern{-0.6}\kern\dimen@\macc@nucleus\rel@kern{0.4}\kern\dimen@}%
      \advance\dimen@0.4\dimexpr\macc@kerna
      \let\final@kern#2%
      \ifdim\dimen@<\z@ \let\final@kern1\fi
      \if\final@kern1 \kern-\dimen@\fi
    \else
      \overline{\rel@kern{-0.6}\kern\dimen@#1}%
    \fi
  }%
  \macc@depth\@ne
  \let\math@bgroup\@empty \let\math@egroup\macc@set@skewchar
  \mathsurround\z@ \frozen@everymath{\mathgroup\macc@group\relax}%
  \macc@set@skewchar\relax
  \let\mathaccentV\macc@nested@a
  \if#31
    \macc@nested@a\relax111{#1}%
  \else
    \def\gobble@till@marker##1\endmarker{}%
    \futurelet\first@char\gobble@till@marker#1\endmarker
    \ifcat\noexpand\first@char A\else
      \def\first@char{}%
    \fi
    \macc@nested@a\relax111{\first@char}%
  \fi
  \endgroup
}
\makeatother

\newcommand{\I}{\mathcal{I}}

\def\cH{{\mathcal H}}

\def\ZZ{{\mathbf Z}}

\def\CC{{\mathbf C}}

\def\QQ{{\mathbf Q}}

\newtheorem*{thmA'}{Theorem~A$^\prime$}

\begin{document}

\vspace{\baselineskip}

\title{On $k$-rational and $k$-Du Bois local complete intersections}

\author[M. Musta\c{t}\v{a}]{Mircea~Musta\c{t}\u{a}}
\address{Department of Mathematics, University of Michigan, 530 Church Street,
Ann Arbor, MI 48109, USA}
\email{{\tt mmustata@umich.edu}}

\author[M.~Popa]{Mihnea~Popa}
\address{Department of Mathematics, Harvard University, 
1 Oxford Street, Cambridge, MA 02138, USA} 
\email{{\tt mpopa@math.harvard.edu}}

\thanks{MM was partially supported by NSF grants DMS-2301463 and DMS-1952399, and MP by NSF grant DMS-2040378.}


\begin{abstract}
We show that local complete intersections with $k$-rational singularities are $k$-Du Bois, as a consequence 
of an injectivity theorem for the duals of the associated graded quotients of the Du Bois complex.
For hypersurfaces, we characterize $k$-rationality in terms of the minimal exponent. We also establish 
some local vanishing results for $k$-rational and $k$-Du Bois singularities. Some of these results have also 
been independently obtained in \cite{FL2}.
\end{abstract}

\maketitle

\makeatletter
\newcommand\@dotsep{4.5}
\def\@tocline#1#2#3#4#5#6#7{\relax
  \ifnum #1>\c@tocdepth 
  \else
    \par \addpenalty\@secpenalty\addvspace{#2}%
    \begingroup \hyphenpenalty\@M
    \@ifempty{#4}{%
      \@tempdima\csname r@tocindent\number#1\endcsname\relax
    }{%
      \@tempdima#4\relax
    }%
    \parindent\z@ \leftskip#3\relax
    \advance\leftskip\@tempdima\relax
    \rightskip\@pnumwidth plus1em \parfillskip-\@pnumwidth
    #5\leavevmode\hskip-\@tempdima #6\relax
    \leaders\hbox{$\m@th
      \mkern \@dotsep mu\hbox{.}\mkern \@dotsep mu$}\hfill
    \hbox to\@pnumwidth{\@tocpagenum{#7}}\par
    \nobreak
    \endgroup
  \fi}
\def\l@section{\@tocline{1}{0pt}{1pc}{}{\bfseries}}
\def\l@subsection{\@tocline{2}{0pt}{25pt}{5pc}{}}
\makeatother

\section{Introduction}
Rational and Du Bois singularities  are two classes of singularities that play an important role in complex algebraic geometry. The notion of \emph{rational singularities}
is of a cohomological nature: a complex algebraic variety $Z$ has this property if for some (any) resolution of singularities $\mu\colon \widetilde{Z}\to Z$, the canonical morphism $\shO_Z \to \derR\mu_*\shO_{\widetilde{Z}}$ is an isomorphism. This is a very useful condition since, if satisfied, the cohomology of a locally free sheaf on $Z$ is isomorphic to that of its pull-back to $\widetilde{Z}$. On the other hand, the notion of \emph{Du Bois singularities} is of a Hodge-theoretic nature. 
Recall that building on Deligne's work \cite{Deligne}, Du Bois associated to any complex algebraic variety $Z$ a filtered complex $\underline{\Omega}_Z^{\bullet}$,
nowadays called the \emph{Du Bois complex} of $Z$. When $Z$ is smooth, this is just the de Rham complex of $Z$, with its ``stupid" filtration, but when $Z$ is singular,
it turns out to have better cohomological properties than the de Rham complex. The $i$-th graded piece $\underline{\Omega}_Z^i$ of the Du Bois complex (suitably shifted) is an element in the bounded derived category of coherent sheaves on $Z$ and it carries a canonical morphism $\Omega_Z^i\to \underline{\Omega}_Z^i$, which is an isomorphism when
$Z$ is smooth. The variety $Z$ has \emph{Du Bois singularities} if the canonical morphism $\shO_Z\to \underline{\Omega}_Z^0$ is an isomorphism. This condition 
is important in birational geometry since projective varieties with Du Bois singularities satisfy a version of Kodaira vanishing. It was shown by Steenbrink
\cite{Steenbrink2}  in the case of isolated singularities, and in full generality by Kov\'acs \cite{Kovacs1} and also later by Saito \cite{Saito-HC}, that rational singularities
are Du Bois.

Recently, a systematic study of natural refinements of these two standard classes of singularities has been taking shape, guided especially by developments of a Hodge theoretic and $\Dmod$-module theoretic flavor.
On one hand, the papers \cite{MOPW} and \cite{Saito_et_al} introduced and studied the notion of \emph{$k$-Du Bois singularities} for hypersurfaces (the terminology appeared in the latter) as a natural extension of the concept of Du Bois singularities. The definition, which makes sense for an arbitrary variety $Z$, is that the natural morphisms 
$$\Omega_Z^i \longrightarrow \underline\Omega_Z^i$$
are isomorphisms for $0 \le i \le k$. The results in \emph{loc.}~\emph{cit.} were extended to local complete intersections in \cite{MP2}.

On the other hand, as defined in \cite{FL1} (cf. also \cite{FL3}) for normal isolated singularities, and communicated to us by 
R. Laza in general, rational singularities also admit a natural refinement: a variety $Z$ has \emph{$k$-rational singularities} if for any 
resolution of singularities $\mu \colon \widetilde{Z} \to Z$ that is an isomorphism over the smooth locus of $Z$ and such that the reduced inverse image $D$ of the singular locus is a simple normal crossing divisor,
the canonical morphisms
$$\Omega_Z^i\to \derR\mu_*\Omega^i_{\widetilde{Z}}(\log\,D)$$
are isomorphisms for all $0\leq i\leq k$. Note that for $k=0$ we recover the classical notions of Du Bois and rational singularities.

The study we undertake here is motivated by two points. First, since we know that rational singularities are Du Bois, it is natural to ask whether this persists for the higher versions.  Second, in the case of hypersurfaces both notions have other possible (numerical) refinements in terms of minimal exponents; the natural question, already approached for isolated singularities in \cite{FL1},  is whether they coincide with the notions defined above. 

In this paper, we answer some of these questions positively for local complete intersections, and others, when the $V$-filtration and minimal exponents are involved, only in the case of hypersurfaces. While writing it, we learned that some of the results we obtain have also been arrived at independently in \cite{FL2} and \cite{Saito-AppendixFL}, following work in the case of isolated singularities in \cite{FL1}, \cite{FL3}; see below. At the moment  it seems quite hard to say much beyond the case of local complete intersections. 

In what follows we always work over $\CC$, and $X$ is an irreducible $n$-dimensional smooth algebraic variety.

\medskip

\noindent
{\bf Results for local complete intersections.} 
The main result we obtain in this context is an injectivity theorem for the cohomologies of the Grothendieck duals of the various graded quotients of the Du Bois complex, whose proof relies on the study of the Hodge filtration on local cohomology in \cite{MP2}.

\begin{intro-theorem}\label{injectivity}
Let $Z$ be an algebraic variety which is locally a complete intersection, and let 
$k$ be a nonnegative integer such that $Z$ has $(k-1)$-Du Bois singularities.  Then the morphism 
$$\derR \mathcal{H}om_{\shO_Z} (\underline{\Omega}_Z^k, \omega_Z) \to \derR \mathcal{H}om_{\shO_Z} (\Omega_Z^k, \omega_Z)$$ 
in the derived category of coherent sheaves on $Z$, obtained by dualizing the canonical morphism 
$\Omega_Z^k \to \underline{\Omega}_Z^k$, is injective at the level of cohomology.
\end{intro-theorem}

When $k = 0$ the hypothesis is vacuous, and the statement is known to hold for an arbitrary variety $Z$ (meaning not necessarily a local complete intersection). This result was shown by Kov\'acs and Schwede \cite[Theorem 3.3]{KoS}, with a somewhat stronger version obtained by different means in \cite[Theorem A]{MP2}, and has proven to be useful for a whole range of applications.

 A quick consequence of Theorem \ref{injectivity} that we derive here is the natural higher analogue, for local complete intersections, of the 
 well-known fact discussed above that rational singularities are Du Bois.
  
\begin{intro-theorem}\label{main2}
Let $Z$ be an algebraic variety which is locally a complete intersection. If $Z$ has $k$-rational singularities, then $Z$ has $k$-Du Bois singularities.
\end{intro-theorem}

This result was also obtained in a different fashion by Friedman and Laza in \cite{FL2}, and previously in the case of 
 isolated singularities in \cite{FL1} (cf. also \cite{FL3}).

An interesting consequence of Theorem \ref{main2} regards duality for the graded pieces of the Du Bois complex. It is not hard to show, see Proposition \ref{duality-DB} below,  that for every $d$-dimensional irreducible variety $Z$ and every $k \ge 0$ there is a natural morphism
$$\psi_k \colon \underline{\Omega}_Z^k \longrightarrow \derR \mathcal{H}om_{\shO_Z} \big( \underline{\Omega}_Z^{d-k}, \omega_Z^\bullet [-d]\big),$$
where $\omega_Z^\bullet$ is the dualizing complex of $Z$. This is however usually not an isomorphism, and in fact we show that this property is precisely what makes
the difference between $k$-Du Bois and $k$-rational singularities:

\begin{intro-corollary}\label{krat-duality}
If $Z$ is an irreducible $d$-dimensional local complete intersection variety, then 
$Z$ has $k$-rational singularities if and only if it has $k$-Du Bois singularities and
the canonical morphism
$$\psi_k \colon \underline{\Omega}_Z^k \longrightarrow \derR \mathcal{H}om_{\shO_Z} ( \underline{\Omega}_Z^{d-k}, \omega_Z)$$
is an isomorphism.
\end{intro-corollary}

By duality, in the setting of Corollary \ref{krat-duality} one can also compute $\underline{\Omega}_Z^{d-k}$ as the derived 
dual of $\Omega_Z^k$; see Corollary \ref{dual-dual}.

On a different note, since the work of Steenbrink, see e.g. \cite{Steenbrink}, it has been known that the graded pieces of the Du Bois complex are closely related to direct images of sheaves of forms with log poles on log resolutions. In this direction, we record the following local vanishing theorem, which is mainly a consequence of our results in \cite{MP2}, with some additions from \cite{GKKP}; this gives a positive answer to a question of R. Laza.
A related result, Theorem \ref{thm_vanishing} below, will be a crucial technical step towards a finer understanding of $k$-rational singularities of hypersurfaces. We again take $\mu \colon \widetilde{Z} \to Z$ to be a resolution that is an isomorphism over the smooth locus of $Z$, and such that the reduced preimage of the singular locus of $Z$ is a simple normal crossing divisor $D$.

\begin{intro-theorem}\label{DB-vanishing}
If $Z$ is a $d$-dimensional local complete intersection variety which has $k$-Du Bois singularities and is normal,\footnote{Note that $Z$ is automatically normal if $k\geq 1$.} then 
$$R^q\mu_*\Omega^p_{\widetilde{Z}}(\log\,D)(-D)=0$$ 
in each of the following cases:
\begin{enumerate}
\item[i)] $p+q\geq d +1$;
\item[ii)] $p\leq k$, $q \ge 1$, and $ p + q \ge d - 2k +1$;
\item[iii)] $p=k+1$ and $q=d-k-1$;
\item[iv)] $q \ge \max \{ d- k - 1, 1\}$.
\end{enumerate}
\end{intro-theorem}

When the singularities of $Z$ are isolated, a different approach to such vanishing is provided in
\cite[Theorem 3.6]{FL1}.

\noindent
{\bf Results for hypersurfaces.}
In the case of hypersurfaces, one can obtain stronger results; the key technical tool allowing for this is the minimal exponent, and 
especially its connection with the $V$-filtration of Kashiwara and Malgrange.

Let $Z$ be a hypersurface in $X$. Its minimal exponent $\widetilde{\alpha}(Z)$ is the negative of the greatest root of the reduced Bernstein-Sato polynomial $b_Z(s) / (s+1)$; see e.g. \cite[\S6]{MP3} for a general discussion of this singularity invariant. The minimal exponent is related to the log canonical threshold by the formula ${\rm lct}(X, Z) = \min\{\widetilde{\alpha}(Z), 1\}$.
M. Saito showed in \cite{Saito-B} that $Z$ has rational singularities if and only if $\widetilde{\alpha}(Z)> 1$.
This can be extended to the following:

\begin{intro-theorem}\label{thm_main}
If $Z$ is a hypersurface in $X$,  and $k$ is a nonnegative integer, then $Z$ has $k$-rational singularities 
if and only if $\widetilde{\alpha}(Z)>k+1$.
\end{intro-theorem}

Another proof of this result was obtained independently by Saito \cite{Saito-AppendixFL}. For isolated singularities, it had been established in \cite[Theorem 3.8]{FL1}, cf. also \cite{FL3}.

\noindent
\emph{Note.}
The notion of $k$-rationality for hypersurfaces first appeared in \cite{KL} precisely as the condition 
$\widetilde{\alpha}(Z)>k+1$. Here we are of course using the more natural definition in terms of log resolutions suggested in \cite{FL1}. The point of the theorem is that these two possible generalizations
indeed coincide. 

Combining the main results of  \cite{MOPW} and \cite{Saito_et_al},  we also know that 
$Z$ is $k$-Du Bois if and only if $\widetilde{\alpha} (Z)  \ge k +1$. Together with Theorem \ref{thm_main}, this provides a numerical strengthening of Theorem \ref{main2}, and has the following immediate consequence (cf. also \cite[Conjecture 1.8]{FL2}):

\begin{intro-corollary}\label{drop}
If $Z$ is a hypersurface in $X$ which has $(k+1)$-Du Bois singularities, for some $k\ge 0$, then $Z$ has $k$-rational singularities.
\end{intro-corollary}

At the moment we do not know how to show this last implication in the case of local complete intersections of higher codimension.

We consider an approach that obtains the ``if" part of Theorem \ref{thm_main} as a quick consequence of the following local vanishing theorem, which is the counterpart to Theorem \ref{DB-vanishing}, and uses the same notation. The question whether a vanishing result roughly of this type holds was 
posed to us by R. Laza. Part ii) is in fact one of the main results of \cite{MP1} (the proof we give here is technically similar, but slightly simpler); the main new result is part i). Again, for isolated singularities see also \cite[Theorem 3.8 and Remark 3.9]{FL1}.

\begin{intro-theorem}\label{thm_vanishing}
If $Z$ is a hypersurface in $X$, $k$ is a nonnegative integer, and $\widetilde{\alpha}(Z)>k+1$, then 
$$R^q\mu_*\Omega_{\widetilde{Z}}^p(\log\,D)=0$$ 
in each of the following cases:
\begin{enumerate}
\item[i)] $p\leq k$ and $q\geq 1$;
\item[ii)] $p=k+1$ and $q=n-k-2$. 
\end{enumerate}
\end{intro-theorem}

This theorem is in turn a consequence of an analogous result we establish for log resolutions of the pair $(X, Z)$, as opposed to $Z$ itself; see Theorem \ref{thm_vanishing_log}. The proofs of Theorems \ref{thm_main} and \ref{thm_vanishing} make use of Saito's theory of mixed Hodge modules \cite{Saito-MHM}, especially duality and the description of the Hodge filtration
on the local cohomology $\cH_Z^1(\shO_X)$ in terms of the $V$-filtration.

\noindent 
{\bf Acknowledgements.}
We especially thank R. Laza for outlining the program we address here, and for asking us many relevant questions.
We also thank him and R.~Friedman for sharing an early version of the paper \cite{FL1}, which discusses 
many of these topics in the case of isolated singularities. We would like to thank the anonymous referee for several comments and suggestions
on a previous version of this article.

\section{Definitions and background}
In what follows we work over the field $\CC$ of complex numbers. By a \emph{variety} we mean a reduced, separated scheme of finite type over $\CC$ (not necessarily irreducible).  For a variety $Z$, we denote by $Z_{\rm sing}$ its singular locus and by $Z_{\rm sm}$ the complement $Z \smallsetminus Z_{\rm sing}$.

\subsection{Log resolutions}
We will deal with two types of resolutions of singularities that we now describe. If $Z$ is an irreducible variety, then by a \emph{log resolution}
of $Z$ we mean a proper birational morphism $\mu\colon\widetilde{Z}\to Z$, with $\widetilde{Z}$ smooth, such that if $W$ is the complement of the domain of $\mu^{-1}$, then the subset $\mu^{-1}(W)_{\rm red}$ (the \emph{exceptional locus} of $\mu$) is a divisor 
$D$ with simple normal crossings. We say that $\mu$ is a \emph{strong log resolution} of $Z$ if, in addition, we have $W=Z_{\rm sing}$ (in general we only have $Z_{\rm sing}\subseteq W$).
Recall that if $X$ is a smooth variety, then a family of subvarieties $(W_i)$ of $X$ has simple normal crossings if locally on $X$ we can find a system of algebraic coordinates $x_1,\ldots,x_n$ (that is, $dx_1,\ldots,dx_n$
trivialize the cotangent bundle) such that every irreducible component of a $W_i$ is defined by a subset of the coordinates.

Suppose now that $Z$ is a subvariety of the smooth irreducible variety $X$. A \emph{log resolution} of $(X,Z)$ is a proper
morphism $\pi\colon Y\to X$, with $Y$ smooth, which is an isomorphism over $X\smallsetminus Z$, and such that $\pi^{-1}(Z)_{\rm red}$ is a divisor with simple normal crossings. If $Z$ is a hypersurface in $X$,
then we say that $\pi$ is a \emph{strong log resolution} of $(X,Z)$ if, in addition, it is an isomorphism over $X\smallsetminus Z_{\rm sing}$.  In this case, if $Z$ is irreducible and $\widetilde{Z}$ is its strict transform on $Y$, then the induced morphism $\widetilde{Z}\to Z$ is a strong log resolution of $Z$. We note that in both contexts strong log resolutions exist by Hironaka's fundamental theorem. Moreover, a (strong) log resolution of $(X,Z)$ can be obtained as a composition of blow-ups of smooth centers.

Let's assume now that  $Z$ is an irreducible hypersurface in $X$, and establish a connection between 
higher direct images of forms with log poles in the two contexts. Suppose that $W$ is a proper closed subset of $Z$ containing 
$Z_{\rm sing}$, and let
$\pi\colon Y\to X$ be a log resolution of $(X,Z)$ that is an isomorphism over $X\smallsetminus W$ and
is a composition of smooth blow-ups. More precisely, the morphism
$\pi$ factors as
$$Y=X_N\overset{\pi_N}\longrightarrow X_{N-1}\longrightarrow\cdots\longrightarrow X_1\overset{\pi_1}\longrightarrow X_0=X,$$
where each $\pi_j$ with $1\leq j\leq N$ is the blow-up of a smooth irreducible subvariety $W_{j-1}\subseteq X_{j-1}$ that lies over $W$.
We denote by $F_j$ the exceptional divisor of $X_j\to X$ and by $Z_j$ the strict transform of $Z$ on $X_j$.
Moreover, for each $1\leq j\leq N$, we assume that $W_{j-1}$ has simple normal crossings with $Z_{j-1}+F_{j-1}$. We note that such a resolution exists by Hironaka's theorem (we can take, for example, $W=Z_{\rm sing}$). We denote
$$\widetilde{Z}=Z_N,\,\,\,\, E=\pi^*(Z)_{\rm red}=\widetilde{Z}+F_N, \,\,\,\, {\rm and} \,\,\,\, D=F_N\vert_{\widetilde{Z}}.$$
If $\mu\colon \widetilde{Z}\to Z$ is the restriction of $\pi$, then $\mu$ is a strong log resolution of $Z$ and $D=\mu^{-1}(Z_{\rm sing})_{\rm red}$.

\begin{proposition}\label{prop_log_res}
With the above notation, if  $r={\rm codim}_X(W)$, then 
for every $i$ with $0\leq i\leq r-2$ and every $q\geq 1$, we have an isomorphism
\begin{equation}\label{eq11_prop_log_res}
R^q\mu_*\Omega_{\widetilde{Z}}^i(\log D)\simeq R^q\pi_*\Omega_Y^{i+1}(\log\,E).
\end{equation}
Moreover, we have an isomorphism
\begin{equation}\label{eq12_prop_log_res}
\mu_*\Omega^i_{\widetilde{Z}}(\log\,D)\simeq {\rm Coker}\big(\Omega_X^{i+1}\to\pi_*\Omega_Y^{i+1}(\log\,E)\big).
\end{equation}
\end{proposition}

\begin{proof}
The argument is similar to the one in the proof of \cite[Theorem~D]{MP1}, but we include the details for the benefit of the reader. 
We first show that for every $i\leq r-1$, we have
\begin{equation}\label{eq1_prop_log_res}
\pi_*\Omega_Y^i(\log\,F_N)=\Omega_X^i\quad\text{and}\quad R^q\pi_*\Omega_Y^i(\log\,F_N)=0\quad\text{for all}\quad q\geq 1.
\end{equation}
Using the Leray spectral sequence, we see that in order to prove (\ref{eq1_prop_log_res}) it is enough to show that for every $1\leq j\leq N$,
we have
\begin{equation}\label{eq2_prop_log_res}
{\pi_j}_*\Omega_{X_j}^i(\log\,F_j)=\Omega_{X_{j-1}}^i(\log\,F_{j-1})\quad\text{and}\quad R^q{\pi_j}_*\Omega_{X_j}^i(\log\,F_j)=0\quad\text{for all}\quad q\geq 1.
\end{equation}
Let us fix $j$. If $W_{j-1}\subseteq F_{j-1}$, then $F_j=\pi_j^*(F_{j-1})$ and the assertion follows from \cite[Theorem~31.1(i)]{MP0}. On the other hand, if $W_{j-1}\not\subseteq F_{j-1}$, then $W_{j-1}$
is the strict transform of its image in $X$. In particular, we have ${\rm codim}_{X_{j-1}}(W_{j-1})\geq r$ and the assertions in (\ref{eq2_prop_log_res}) for $j$
follow from \cite[Lemma~7.2]{MP1}.

We now consider on $Y$ the residue short exact sequence
$$0\longrightarrow \Omega_Y^{i+1}(\log\,F_N)\longrightarrow \Omega_Y^{i+1}(\log\,E)\longrightarrow\Omega_{\widetilde{Z}}^i(\log\,D)\longrightarrow 0.$$
The long exact sequence for higher direct images together with the formulas in (\ref{eq1_prop_log_res}) imply that for $i\leq r-2$ we have an isomorphism
$$R^q\pi_*\Omega_Y^{i+1}(\log E)\simeq R^q\mu_*\Omega_{\widetilde{Z}}^i(\log\,D)\quad\text{for all}\quad q\geq 1$$
and a short exact sequence
$$0\longrightarrow \Omega_X^{i+1}\longrightarrow \pi_*\Omega_Y^{i+1}(\log\,E)\longrightarrow \mu_*\Omega^i_{\widetilde{Z}}(\log\,D)\longrightarrow 0.$$
This completes the proof of the proposition.
\end{proof}

\begin{lemma}\label{indep_res2}
Let $Z$ be an irreducible variety, 
$\mu\colon \widetilde{Z}\to Z$ a log resolution of $Z$ with $W$ the complement of the domain of $\mu^{-1}$, and let $D$ be the simple normal crossing divisor
on $\widetilde{Z}$ such that $\mu^{-1}(W)_{\rm red}=D$. 
For every $p$ and $q$, the sheaves 
$$R^q\mu_*\Omega_{\widetilde{Z}}^p(\log\,D) \,\,\,\,{\rm and} \,\,\,\,R^q\mu_*\Omega_{\widetilde{Z}}^p(\log\,D) (-D)$$ 
only depend on $W$ (but not on $\mu$). In particular, these sheaves only depend on $Z$ if $\mu$ is assumed to be a strong log resolution.
\end{lemma}

\begin{proof}
Since any two log resolutions  with the same $W$ are dominated 
by a common one, in order to prove the assertion it is enough to show that if $E$ is a reduced simple normal crossing divisor on the smooth $n$-dimensional variety $Y$ and $g\colon \widetilde{Y}\to Y$
is a proper morphism that is an isomorphism over $Y\smallsetminus {\rm Supp}(E)$, with $\widetilde{Y}$ smooth and $F=g^*(E)_{\rm red}$ having simple normal crossings, then for every $p\geq 0$,
we have canonical isomorphisms
\begin{equation}\label{eq_rmk_indep_res3}
\derR g_*\Omega_{\widetilde{Y}}^p(\log\,F)\simeq\Omega_Y^p(\log\,E).
\end{equation}
and
\begin{equation}\label{eq_rmk_indep_res2}
\derR g_*\Omega_{\widetilde{Y}}^p(\log\,F)(-F)\simeq\Omega_Y^p(\log\,E)(-E).
\end{equation}
In fact ($\ref{eq_rmk_indep_res3}$) is already known; see for example \cite[Theorem~31.1(i)]{MP0}.
On the other hand, since $\Omega_Y^n(\log\,E)\simeq\omega_Y(E)$ and $\Omega_{\widetilde{Y}}^n(\log\,F)\simeq\omega_{\widetilde{Y}}(F)$, we have
$$\big(\Omega_Y^p(\log\,E)(-E)\big)^{\vee}\simeq\Omega_Y^{n-p}(\log\,E)\otimes\omega_Y^{-1} \,\,\,\,{\rm and}\,\,\,\,
\big(\Omega_{\widetilde{Y}}^p(\log\,F)(-F)\big)^{\vee}\simeq\Omega^{n-p}_{\widetilde{Y}}(\log\,F)\otimes\omega_{\widetilde{Y}}^{-1},$$
and using Grothendieck duality we obtain
$$\derR {\mathcal Hom}_{\shO_Y}\big(\derR g_*\Omega_{\widetilde{Y}}^p(\log\,F)(-F),\omega_Y\big)\simeq 
\derR g_*\derR {\mathcal Hom}_{\shO_{\widetilde{Y}}}\big(\Omega_{\widetilde{Y}}^p(\log\,F)(-F),\omega_{\widetilde{Y}}\big)$$
$$\simeq \derR g_*\Omega_{\widetilde{Y}}^{n-p}(\log\,F)\simeq \Omega_Y^{n-p}(\log\,E)\simeq \derR {\mathcal Hom}_{\shO_Y}\big(\Omega_Y^p(\log\,E)(-E),\omega_Y\big).$$
The isomorphism in (\ref{eq_rmk_indep_res2}) follows from the fact that $\derR {\mathcal Hom}_{\shO_Y}(-,\omega_Y)$ is a duality. 
\end{proof}

\subsection{$k$-rational singularities}
The following definition is due to Friedman and Laza \cite{FL1} in the case of normal isolated singularities, and was communicated to us by Laza in general.  

\begin{definition}\label{def_k_rat}
Let $Z$ be an irreducible variety and $\mu\colon\widetilde{Z}\to Z$ a strong log resolution, with $\mu^{-1}(Z_{\rm sing})_{\rm red}=D$. 
For an integer $k\geq 0$, we say that $Z$ has \emph{$k$-rational singularities} if the canonical morphisms
\begin{equation}\label{eq_defi_k-ration}
\Omega_Z^i\to \derR\mu_*\Omega^i_{\widetilde{Z}}(\log\,D)
\end{equation}
are isomorphisms for all $0\leq i\leq k$. We say that an arbitrary variety $Z$ has $k$-rational singularities if all its connected components are irreducible, with $k$-rational singularities.
\end{definition}

\begin{remark}\label{rmk_indep_res}
By Lemma \ref{indep_res2}, the definition is independent of the choice of strong log resolution.
\end{remark}

\begin{remark}
For $k=0$, we recover the familiar notion of \emph{rational singularities}.
\end{remark}

\begin{remark}\label{rem:reflexive}
Recall that a coherent sheaf $\shF$ on $Z$ is \emph{reflexive} if the canonical morphism $\shF\to \shF^{\vee\vee}$ is an isomorphism. 
If $Z$ is normal, then $\shF$ is reflexive if and only if it is torsion-free and for every open subsets $V'\subseteq V\subseteq X$
with ${\rm codim}_V(V\smallsetminus V')\geq 2$, the restriction map $\Gamma(V,\shF)\to \Gamma(V',\shF)$ is an isomorphism (see \cite[Proposition~1.6]{Hartshorne}). 
This implies that if $U=Z\smallsetminus Z_{\rm sing}$, with $j\colon U\hookrightarrow X$ the inclusion, and $\shF\vert_U$ is locally free, then $\shF$ is reflexive if and only if
the canonical map $\shF\to j_*(\shF\vert_U)$ is an isomorphism. 
\end{remark}

\begin{remark}\label{rmk_level_zero}
With the notation in Definition~\ref{def_k_rat}, the condition that $Z$ has $k$-rational singularities is equivalent to the vanishings
$$R^q\mu_*\Omega_{\widetilde{Z}}^i(\log\,D)=0\quad\text{for}\quad i\leq k, \,\,q\geq 1$$
and the fact that the canonical morphism 
\begin{equation}\label{eq_rmk_level_zero}
\Omega_Z^i\to\mu_*\Omega_{\widetilde{Z}}^i(\log\,D)
\end{equation}
is an isomorphism for all $i\leq k$. Note that for $i=0$, the morphism (\ref{eq_rmk_level_zero}) is an isomorphism 
if and only if $Z$ is normal. On the other hand, for $i\geq 1$, as a consequence of the extension theorem for differential forms in \cite{KS} we have:
\end{remark}

\begin{lemma}\label{lem:reflexive}
If $Z$ has rational singularities, then the morphism (\ref{eq_rmk_level_zero}) is an isomorphism if and only if $\Omega_Z^i$ is a reflexive sheaf.
\end{lemma}

\begin{proof}
The key point is to show that since $Z$ has rational singularities, $\mu_*\Omega_{\widetilde{Z}}^i(\log\,D)$ is reflexive for every $i$. 
Note first that $\mu_*\omega_{\widetilde{Z}}(D)$ is a reflexive sheaf.
Indeed, the fact that $Z$ has rational singularities implies that $Z$ is normal and Cohen-Macaulay and $\mu_*\omega_{\widetilde{Z}}\simeq\omega_Z$, which is reflexive; see \cite[Theorem~5.10]{KM}. Since the inclusion $\mu_*\omega_{\widetilde{Z}}\hookrightarrow \mu_*\omega_{\widetilde{Z}}(D)$
is an isomorphism over the smooth locus of $Z$ and $\mu_*\omega_{\widetilde{Z}}(D)$ is torsion-free, it is straightforward to deduce that 
$\mu_*\omega_{\widetilde{Z}}(D)$ is reflexive, and in fact isomorphic to $\mu_*\omega_{\widetilde{Z}}$. 
We can thus apply 
\cite[Theorem~1.5]{KS} to deduce that $\mu_*\Omega_{\widetilde{Z}}^i(\log\,D)$ is reflexive for every $i$. 
Since the morphism (\ref{eq_rmk_level_zero}) is an isomorphism over the smooth locus of $Z$, it follows that it is an isomorphism if and only if $\Omega_Z^i$ is reflexive
(see the last assertion in Remark~\ref{rem:reflexive}).
\end{proof}

It is a natural question whether the morphisms (\ref{eq_defi_k-ration}) are isomorphisms when $X$ has $k$-rational singularities,
but $\mu$ is an arbitrary log resolution of $Z$. We now address this.

\begin{proposition}\label{prop_gen_log_res}
Let $Z$ be an irreducible variety, 
$\mu\colon \widetilde{Z}\to Z$ a log resolution of $Z$ with $W$ the complement of the domain of $\mu^{-1}$, and let $D$ be the simple normal crossing divisor
on $\widetilde{Z}$ such that $\mu^{-1}(W)_{\rm red}=D$. We assume that $W\cap Z_{\rm sm}\neq\emptyset$ and let 
$r={\rm codim}_{Z_{\rm sm}}(W\cap Z_{\rm sm})$.
\begin{enumerate}
\item[i)] If $R^q\mu_*\Omega^i_{\widetilde{Z}}({\rm log}\,D)=0$ for $q\geq 1$ and $i\leq k$, then $r>k$.
\item[ii)] If $Z$ has $k$-rational singularities and $r>k$, then the canonical morphism
$\Omega_Z^i\to \derR\mu_*\Omega^i_{\widetilde{Z}}(\log\,D)$ is an isomorphism for every $i\leq k$. 
\end{enumerate}
\end{proposition}

\begin{proof}
We first note that since $\mu^{-1}$ is defined in codimension $1$ on $Z_{\rm sm}$, we have $r\geq 2$.
In order to prove the assertion in i), given any irreducible component $W_0$ of $W$ that intersects $Z_{\rm sm}$, 
we may replace $Z$ by a suitable smooth open subset $U$ such that $U\cap W_0$ is smooth and irreducible.
Therefore we may and will assume that both $Z$ and $W$ are smooth and irreducible.
By Lemma~\ref{indep_res2},
we may assume that $\mu$ is the blow-up of $Z$ along $W$. In this case, we have
$$R^{r-1}\mu_*\Omega^j_{\widetilde{Z}}(\log\,D)\simeq \Omega_W^{j-r}\quad\text{for all}\quad j$$
(see for example \cite[Lemma~4.27]{MP2}). Therefore this is nonzero if $r\leq j\leq \dim(Z)$, hence the assumption in i) implies $r>k$.

In order to prove the assertion in ii), we use again the fact that by Lemma~\ref{indep_res2}, we may 
choose a convenient $\mu$. We thus may and will
assume that $\mu=\varphi\circ\psi$,
as follows. The morphism $\varphi\colon \widetilde{Z_0}\to Z$ is a strong log resolution of $Z$, with $\varphi^{-1}(Z_{\rm sing})_{\rm red}=F$; we denote by
$W_0\subseteq \widetilde{Z}_0$ the union of the strict transforms of the irreducible components of $W$ that intersect $Z_{\rm sm}$. 
The morphism $\psi\colon \widetilde{Z}\to \widetilde{Z}_0$ is a composition of smooth blow-ups 
$$\widetilde{Z}=\widetilde{Z}_N\overset{\psi_N}\longrightarrow\cdots\overset{\psi_2}\longrightarrow\widetilde{Z}_1\overset{\psi_1}\longrightarrow\widetilde{Z_0},$$
where each $\psi_j$ is the blow-up of a smooth, irreducible subvariety $V_{j-1}\subseteq \widetilde{Z}_{j-1}$ that is mapped into $W_0$. Moreover, if 
$E_j\subseteq \widetilde{Z}_j$ is the exceptional divisor of $\widetilde{Z}_j\to \widetilde{Z}_0$ and $F_j$ is the strict transform of $F$ on $\widetilde{Z}_j$,
then we assume that $V_{j-1}$ has simple normal crossings with $F_{j-1}+E_{j-1}$ for $1\leq j\leq N$. The existence of such $\psi$ is again a consequence
of Hironaka's theorem. With this notation, note that $D=E_N+F_N$. 

Since $Z$ has $k$-rational singularities, we know that the canonical morphism
$$\Omega_Z^i\to\derR \varphi_*\Omega^i_{\widetilde{Z}_0}(\log\,F)$$
is an isomorphism for all $i\leq k$. We deduce using the Leray spectral sequence that 
in order to prove the assertion in ii), it is enough to show that for every $i\leq k$ and every $j$, with $1\leq j\leq N$, the canonical morphism
\begin{equation}\label{eq_prop_gen_log_res}
\Omega^i_{\widetilde{Z}_{j-1}}\big(\log (F_{j-1}+E_{j-1})\big)\to\derR {\psi_j}_*\Omega^i_{\widetilde{Z}_j}\big(\log (F_j+E_j)\big)
\end{equation}
is an isomorphism.
If $V_{j-1}\subseteq {\rm Supp}(F_{j-1}+E_{j-1})$, then (\ref{eq_prop_gen_log_res}) is an isomorphism for all $i$,
see  \cite[Theorem~31.1(i)]{MP0}. On the other hand,
if $V_{j-1}\not\subseteq {\rm Supp}(F_{j-1}+E_{j-1})$, then (\ref{eq_prop_gen_log_res}) is an isomorphism if $i\leq {\rm codim}_{\widetilde{Z}_{j-1}}(V_{j-1})-1$,
see \cite[Lemma~4.27]{MP2}.
Note that since $V_{j-1}$ is not contained in the exceptional locus of $\widetilde{Z}_{j-1}\to Z$, it follows that it is the strict transform of its image in $Z$,
which is contained in an irreducible component of $W$ meeting $Z_{\rm sm}$. We thus have 
$${\rm codim}_{\widetilde{Z}_{j-1}}(V_{j-1})-1\geq r-1\geq k\geq i,$$
hence (\ref{eq_prop_gen_log_res}) is an isomorphism in this case as well. This completes the proof of the proposition.
\end{proof}

\subsection{Brief review of  filtered $\Dmod$-modules}
We will make use of the theory of mixed Hodge modules, for which we refer to \cite{Saito-MHM}. Let $X$ be a smooth, irreducible, $n$-dimensional complex algebraic variety.
Recall that every (mixed) Hodge module on $X$ has an underlying 
filtered $\Dmod_X$-module $(\Mmod,F)$; the  filtration $F$ is called the \emph{Hodge filtration}. For an integer $q$, the \emph{Tate twist} $\Mmod(q)$ has the same underlying
$\Dmod_X$-module, but the filtration is given by $F_p\Mmod(q)=F_{p-q}\Mmod$ for all $p\in\ZZ$. 

We will use both left and right Hodge modules. Recall that there is an equivalence between the categories of such objects, such that 
if $(\Mmod^r,F)$ is the right filtered $\Dmod_X$-module that corresponds to $(\Mmod,F)$, then there is an isomorphism of $\shO_X$-modules
 $\Mmod^r\simeq\omega_X\otimes_{\shO_X}\Mmod$
that induces for every $p$ an isomorphism 
$$F_{p-n}\Mmod^r\simeq\omega_X\otimes_{\shO_X} F_p\Mmod.$$ 
For example, the left version of the trivial (pure) Hodge module
${\mathbf Q}_X^H[n]$ is $(\shO_X,F)$, where ${\rm gr}^F_i \shO_X =0$ for $i\neq 0$, and the corresponding right filtered $\Dmod_X$-module  is
$(\omega_X,F)$, where ${\rm gr}^F_i \omega_X =0$ for $i\neq -n$.

Recall that for every filtered left  Hodge module $(\Mmod, F)$ on $X$ and every $i\in\ZZ$, the complex ${\rm Gr}^F_{i}{\rm DR}_X(\Mmod)$ is given by 
$$0\to {\rm Gr}^F_i \Mmod \to \Omega_X^{1}\otimes_{\shO_X}{\rm Gr}^F_{i+1} \Mmod \to \cdots \to \omega_X
\otimes_{\shO_X}{\rm Gr}^F_{i+n} \Mmod \to 0,$$
placed in cohomological degrees $-n,\ldots,0$. The same complex can be written in terms of the corresponding right Hodge module $(\Mmod^r,F)$ as

$$0\to \wedge^n T_X\otimes_{\shO_X}{\rm Gr}^F_{i-n} \Mmod^r \to 
\wedge^{n-1} T_X\otimes_{\shO_X}{\rm Gr}^F_{i-n+1} \Mmod^r \to\cdots\to {\rm Gr}^F_i \Mmod^r \to 0.$$

\subsection{The Du Bois complex and $k$-Du Bois singularities}
Recall that the \emph{Du Bois complex} of a complex variety $W$ is an element
$(\underline{\Omega}_W^{\bullet},F)$ in a suitable filtered derived category. We will only be interested in its shifted graded pieces
$$\underline{\Omega}_W^p:={\rm Gr}^p_F\underline{\Omega}_W^{\bullet}[p] \,\,\,\,\,\,{\rm for}\,\,\,\,p \ge 0,$$
which are objects in the bounded derived category of coherent sheaves $D^b_{\rm coh}(W)$.
For an introduction to the Du Bois complex and for further references, we refer to \cite[Chapter V.3]{GNPP} or \cite[Chapter~7.3]{PS}.
We note that $\shH^i(\underline{\Omega}_W^p)=0$ for all $i<0$ and $\shH^0(\underline{\Omega}_W^p)$ is torsion-free; this follows, for example, from the description in \cite[Theorem~7.12]{Huber}.

For every $W$ as above and every $p$, there is a canonical morphism $\sigma_p\colon \Omega_W^p\to \underline{\Omega}_W^p$ that is an isomorphism over the smooth
locus of $W$. Recall that by definition, $W$ has Du Bois singularities if $\sigma_0$ is an isomorphism. More generally, following \cite{Saito_et_al}, we will say that $W$
has \emph{$k$-Du Bois singularities} if $\sigma_p$ is an isomorphism for all $p$, with $0\leq p\leq k$. 

\begin{remark}\label{rmk_factorization}
Suppose that $W$ is an irreducible variety and $\mu\colon\widetilde{W}\to W$ is any resolution of singularities.
By functoriality of the Du Bois complex, for every $p$ we have a canonical map
$$\underline{\Omega}_W^p\to \derR \mu_*\underline{\Omega}_{\widetilde{W}}^p=\derR \mu_*\Omega_{\widetilde{W}}^p$$
whose composition with $\sigma_p$ is the canonical morphism $\Omega_W^p\to\derR \mu_*\Omega_{\widetilde{W}}^p$.
Indeed, this follows using the next remark,  since $\mu_*\Omega_{\widetilde{W}}^p$ is torsion-free and
the assertion certainly holds on the locus in $W$ over which $\mu$ is an isomorphism.
\end{remark}

The following general remark is used repeatedly throughout the paper.

\begin{remark}\label{identify_morphisms}
Let $W$ be a variety and let $u\colon{\mathcal F}\to {\mathcal P}$ be a morphism in $D^b_{\rm coh}(W)$,
where ${\mathcal F}$ is a coherent sheaf and $\shH^i({\mathcal P})=0$ for all $i<0$. If $\shH^0({\mathcal P})$ is torsion-free
and $u\vert_V=0$ for some open dense subset $V$ of $W$, then $u=0$. Indeed, giving such a morphism $u$ is equivalent to giving the corresponding morphism of coherent sheaves ${\mathcal F}\to \shH^0({\mathcal P})$, and the assertion is clear. 
\end{remark}

\noindent
{\bf Local complete intersections.}
We next restrict to the case of a local complete intersection $Z$ of pure codimension $r$ in $X$, and review the connection established in \cite{MP2} between the Du Bois complex of $Z$ and the Hodge filtration on the local cohomology sheaf $\shH^r_Z(\shO_X)$. This sheaf has a $\Dmod_X$-module structure which underlies a (left) mixed Hodge module. In particular, it carries a Hodge filtration $F$, and for every integer $k$ we have an isomorphism
\begin{equation}\label{isom-DB}
\underline{\Omega}_Z^k\simeq \derR \shH om_{\shO_X}\big({\rm Gr}^F_{k-n}{\rm DR}_X\shH_Z^r(\shO_X),\omega_X\big)[k+r]
\end{equation}
provided by \cite[Proposition 5.5]{MP2}.

The sheaf $\shH^r_Z(\shO_X)$  also carries a more elementary filtration, also compatible with the order filtration on $\Dmod_X$, namely the \emph{Ext filtration}, given for $p\geq 0$ by 
$$E_p\shH^r_Z(\shO_X)=\big\{u\in \shH^r_Z(\shO_X)\mid \I_Z^{p+1}u=0\big\}=
{\rm Im} ~\big[ \shE xt^r_{\shO_X} \big(\shO_X/ \I_Z^{p+1}, \shO_X\big) \to 
\cH_Z^r(\shO_X) \big],$$
where $\I_Z$ is the ideal defining $Z$ in $X$. It is known that for each $p$ we have 
$F_p\shH^r_Z(\shO_X) \subseteq E_p\shH^r_Z(\shO_X)$; see \cite[\S3.1 and \S3.3]{MP2}.

The  \emph{singularity level} of the Hodge filtration on $\shH^r_Z \shO_X$ is defined as 
$$p (Z) := {\rm sup}\{k ~| ~ F_k \shH^r_Z \shO_X = E_k \shH^r_Z \shO_X \},$$
with the convention that $p (Z) = -1$ if there are no such $k$. For a general study of this invariant, see 
\cite[\S3.3]{MP2}. We only mention a few points: first, we have that $p(Z)=\infty$ if and only if $Z$ is smooth. 
Moreover, when $Z$ is singular, we have
\begin{equation}\label{eqn:ineq}
{\rm codim}_Z (Z_{\rm sing}) \ge 2 p(Z) + 1.
\end{equation}
In general, the singularity level $p(Z)$ only depends on $Z$ (not on the ambient smooth variety $X$). 
If $Z$ is a hypersurface in $X$, then 
 $p(Z)=\lfloor \widetilde{\alpha}(Z)\rfloor-1$, where $\widetilde{\alpha}(Z)$ is the minimal exponent of $Z$, recalled in the  
 Introduction.

The main result proved in \emph{loc}.~\emph{cit}. relates $p(Z)$ to the complexity of the Du Bois complex of $Z$, extending the 
hypersurface case treated in  \cite{MOPW} and \cite{Saito_et_al}.

\begin{theorem}[{\cite[Theorem F]{MP2}}]\label{thm-DB-main}
For every nonnegative integer $p$, we have $p(Z)\geq k$ if and only if $Z$ has $k$-Du Bois singularities.
\end{theorem}

\section{Results for local complete intersections}
We recall that $X$ denotes a smooth irreducible $n$-dimensional variety. All throughout this chapter, 
$Z$ will be a  local complete intersection closed subvariety of $X$, of pure codimension $r$.

\subsection{Injectivity and hierarchy of singularities}
In this section we prove the main statement for local complete intersections, the injectivity Theorem \ref{injectivity}, and 
deduce Theorem \ref{main2}, stating that $k$-rational local complete intersection singularities are $k$-Du Bois.

In preparation for the proof of Theorem \ref{injectivity}, we consider for each $0 \le k \le \dim Z$ the complex $C_k^\bullet$ defined by
$$0\to {\rm Sym}^k N_{Z/X}^{\vee}\to \Omega_X^{1}\otimes_{\shO_X}{\rm Sym}^{k-1} N_{Z/X}^{\vee}\to\cdots\to 
\Omega_X^{k-1}\otimes_{\shO_X} N_{Z/X}^{\vee}\to \Omega_X^k\otimes_{\shO_X}\shO_Z\to 0$$
and placed in cohomological degrees $-k,\ldots,0$, obtained by truncating the generalized Eagon-Northcott complex $D_{n-r-k}$ (see \cite[Chapter~2.C]{Bruns}) associated to the canonical morphism 
$$g\colon T_X\vert_Z\to N_{Z/X},$$ 
keeping the first $k+1$ terms, and then suitably translating. We have a canonical isomorphism $\cH^0(C_k^{\bullet})\simeq \Omega_Z^k$ 
induced by the exact sequence
$$N_{Z/X}^{\vee}\to \Omega^1_X\vert_Z\to \Omega^1_Z\to 0.$$
It is shown in \cite[\S5.2]{MP2} (right before the proof of Theorem F) that if ${\rm codim}_Z (Z_{\rm sing}) \ge k$, then $C_k^\bullet$ is a resolution of $ \Omega_Z^k$.

\medskip

We are now ready to prove the injectivity theorem for the duals of the graded pieces of the Du Bois complex of a local complete intersection.

\begin{proof}[Proof of Theorem \ref{injectivity}]
Since $p (Z) \ge k -1$, using ($\ref{eqn:ineq}$) we deduce that 
$${\rm codim}_Z (Z_{\rm sing}) \ge 2k -1.$$
We conclude that the complex $C_k^\bullet$ is a locally free resolution of $\Omega_Z^k$ over $\shO_Z$; for $k = 0$ 
this is obvious, while for $k \ge 1$ the inequality above implies in particular that ${\rm codim}_Z (Z_{\rm sing}) \ge k$,
hence we can apply the  discussion before the start of the proof. 
We can therefore compute $\derR \mathcal{H}om_{\shO_Z} (\Omega_Z^k, \omega_Z)$ by applying the functor 
$\mathcal{H}om_{\shO_Z} (- , \omega_Z)$ to $C_k^\bullet$; this leads to the complex 
$$
0\to \Omega_X^{n-k}\otimes\omega_Z\otimes\omega_X^{-1}\to \Omega_X^{n-k+1}\otimes N_{Z/X}\otimes\omega_Z\otimes\omega_X^{-1}\to
\cdots\to \Omega_X^n\otimes {\rm Sym}^k N_{Z/X}\otimes\omega_Z\otimes\omega_X^{-1}\to 0
$$
placed in cohomological degrees $-k,\ldots,0$. But this turns out to be precisely the complex 
${\rm Gr}^E_{k-n}{\rm DR}_X\shH^r_Z(\shO_X)$, 
i.e. the associated graded of the de Rham complex of the $\Dmod_X$-module $\shH^r_Z(\shO_X)$ with respect 
to the Ext filtration; see \cite[Lemma 3.21]{MP2}, as well as \cite[\S5.2]{MP2}, especially the discussion before formula (5.5). In other words, under our assumption on the minimal exponent, we have a natural isomorphism 
$$\derR \mathcal{H}om_{\shO_Z} (\Omega_Z^k, \omega_Z) \simeq {\rm Gr}^E_{k-n}{\rm DR}_X\shH^r_Z(\shO_X)[-k].$$

On the other hand, by dualizing the isomorphism in ($\ref{isom-DB}$), we also have 
$$\derR \mathcal{H}om_{\shO_X} (\underline{\Omega}_Z^k, \omega_X) \simeq {\rm Gr}^F_{k-n}{\rm DR}_X\shH^r_Z(\shO_X)[-k-r],$$
where this time the associated graded is taken with respect to the Hodge filtration. Using Grothendieck duality for 
the inclusion $Z\hookrightarrow X$, we deduce an isomorphism 
$$\derR \mathcal{H}om_{\shO_Z} (\underline{\Omega}_Z^k, \omega_Z) \simeq 
{\rm Gr}^F_{k-n}{\rm DR}_X\shH^r_Z(\shO_X)[-k]$$
in $D^b_{\rm coh}(X)$.
We claim that via these identifications, the canonical morphism 
$$\varphi_k\colon {\rm Gr}^F_{k-n}{\rm DR}_X\shH^r_Z(\shO_X)\to {\rm Gr}^E_{k-n}{\rm DR}_X\shH^r_Z(\shO_X)$$
induced by the inclusion of the Hodge filtration into the Ext filtration is the dual of the canonical morphism $\Omega_Z^k\to \underline{\Omega}_Z^k$. 
Indeed, since $\shH^0(\underline{\Omega}_Z^k)$ is torsion-free, it follows from Remark~\ref{identify_morphisms} that it is enough to check this on the
complement of the singular locus of $Z$. We may thus assume that $Z$ is smooth, in which case the assertion is straightforward to check.

Note that $\varphi_k$ is a morphism of complexes placed in nonpositive degrees. 
Moreover, by definition of the
singularity level, since $p (Z) \ge k-1$, it follows $\varphi_k$ is an injective morphism of complexes and its cokernel is concentrated in degree $0$. It is then immediate to check that $\cH^i(\varphi_k)$ is an isomorphism for all $i\neq 0$ and $\cH^0(\varphi_k)$ is injective.
\end{proof}

The fact that $k$-rational implies $k$-Du Bois in our setting is now an easy application.

\begin{proof}[Proof of Theorem~\ref{main2}]
Since the connected components of $Z$ are irreducible, we may and will assume that $Z$ is irreducible. 
We prove the result by induction on $k\geq 0$. The canonical morphism
$\Omega_Z^p\to \derR\mu_*\Omega_{\widetilde{Z}}^k(\log D)$ factors as
$$\Omega_Z^k \overset{\sigma_k}{\longrightarrow} \underline{\Omega}_Z^k \overset{\tau_k}{\longrightarrow} \derR \mu_* \Omega_{\widetilde{Z}}^k 
(\log D)$$
(see Remark~\ref{rmk_factorization}). Dualizing this, 
we obtain the composition
$$ \derR \mathcal{H}om_{\shO_Z} (\derR \mu_* \Omega_{\widetilde{Z}}^k 
(\log D), \omega_Z) \overset{\tau'_k}{\longrightarrow}  \derR \mathcal{H}om_{\shO_Z} ( \underline{\Omega}_Z^k, \omega_Z) 
\overset{\sigma'_k}{\longrightarrow}  \derR \mathcal{H}om_{\shO_Z} (\Omega_Z^k , \omega_Z).$$

Our hypothesis gives that $\tau_k \circ \sigma_k$ is an isomorphism, and therefore so  is $\sigma_k' \circ \tau'_k$. On the other hand, since $Z$ is also $(k-1)$-rational, 
it is $(k-1)$-Du Bois by induction (this is vacuous if $k=0$). Therefore Theorem \ref{injectivity} applies, to the effect that $\sigma'_k$ induces injective 
maps in cohomology. These maps are also surjective, hence $\sigma'_k$ is an isomorphism, and so is $\sigma_k$. Therefore $Z$ has
 $k$-Du Bois singularities. 
\end{proof}

\subsection{Duality}\label{scn:duality}
The $k$-rationality condition has interesting consequences regarding duality for the graded pieces of the Du Bois complex. 
We start with a general construction:

\begin{proposition}\label{duality-DB}
For every irreducible variety $Z$ of dimension $d$ and every $k \ge 0$, there is a canonical morphism 
$$\psi_k \colon \underline{\Omega}_Z^k \longrightarrow \derR \mathcal{H}om_{\shO_Z} \big(\underline{\Omega}_Z^{d-k}, \omega_Z^\bullet [-d]\big)$$
in the bounded derived category of coherent sheaves on $Z$.
\end{proposition}

\begin{proof}
Let $f \colon Y \to Z$ be any resolution of singularities. The functoriality of the Du Bois complex gives a morphism 
$$\alpha_k \colon \underline{\Omega}_Z^k \longrightarrow \derR f_* \Omega_Y^k$$
for each $k$. On the other hand, on $Y$ we have an isomorphism 
$$\tau^Y_k\colon \Omega_Y^k\overset{\simeq}\longrightarrow
\derR\mathcal{H}om_{\shO_Y}\big(\Omega_Y^{d-k},\omega_Y^{\bullet}[-d]\big)$$ and we get an isomorphism $\beta_k$ on $Z$ as the composition
$$\derR f_* \Omega_Y^k \overset{\simeq }\longrightarrow \derR f_*\derR\mathcal{H}om_{\shO_Y}\big(\Omega_Y^{d-k},\omega_Y^{\bullet}[-d]\big)
\overset{\simeq}\longrightarrow
\derR \mathcal{H}om_{\shO_Z}
\big(\derR f_* \Omega_Y^{d-k}, \omega_Z^\bullet [-d]\big),$$
where the first isomorphism is $\derR f_*(\tau^Y_k)$ and the second isomorphism is provided by relative duality for $f$. 
Finally, taking the Grothendieck dual of $\alpha_{d-k}$ provides a morphism
$$\gamma_k \colon \derR \mathcal{H}om 
\big(\derR f_* \Omega_Y^{d-k}, \omega_Z^\bullet [-d]\big) \longrightarrow 
\derR \mathcal{H}om 
\big(\underline{\Omega}_Z^{d-k}, \omega_Z^\bullet [-d]\big).$$
We define
$$\psi_k := \gamma_k \circ \beta_k\circ \alpha_k.$$

We need to show that this is independent of the choice of resolution. Since any two resolutions are dominated by a third one, it is enough to show that  if $f$ is as above and $g\colon W\to Y$ is
a proper birational morphism, with $W$ smooth, then the morphisms $\psi_k$ corresponding to $f$ and $h=f\circ g$ coincide.
This follows easily from the definitions once we know that the following diagram is commutative
$$
\begin{tikzcd}
\derR f_* \Omega_Y^k\dar\rar & \derR \mathcal{H}om \big(\derR f_* \Omega_Y^{d-k}, \omega_Z^\bullet [-d]\big)\\
\derR h_*\Omega_W^k\rar & \derR \mathcal{H}om\big(\derR h_* \Omega_W^{d-k}, \omega_Z^\bullet [-d]\big)\uar,
\end{tikzcd}
$$
in which the top horizontal map is the $\beta_k$ with respect to $f$ and the bottom map is $\beta_k$ with respect to $h$. 
This commutativity follows from the functoriality of relative duality and its compatibility with composition of proper morphisms, together with the commutativity
of the diagram
$$
\begin{tikzcd}
\Omega_Y^k\rar{\tau^Y_k}\dar & \derR\mathcal{H}om_{\shO_Y}\big(\Omega_Y^{d-k},\omega_Y^{\bullet}[-d]\big) \\
\derR g_*\Omega_W^k\rar{\derR g_*(\tau^W_k)}& \,\,\,\,\derR g_*\derR\mathcal{H}om_{\shO_W}\big(\Omega_W^{d-k},\omega_W^{\bullet}[-d]\big)\uar.
\end{tikzcd}
$$
The latter follows in turn from the fact that it trivially holds over any open subset over which $g$ is an isomorphism (note that we are comparing two morphisms between
vector bundles). This completes the proof of the proposition.
\end{proof}

It is interesting to understand under what assumptions $\psi_k$ is an isomorphism, as in the case of smooth varieties.
Note that this requires assumptions on the singularities. For instance, when $k = 0$, even when $Z$ is Cohen-Macaulay and Du Bois, $\psi_0$ being an isomorphism is equivalent to the condition $\mu_* \omega_{\widetilde Z} \simeq \omega_Z$, where $\mu \colon \widetilde{Z} \to Z$ is a resolution of singularities; in other words, it is equivalent to $Z$ having rational singularities. More generally, we now show that the condition that $\psi_k$ is an isomorphism is precisely the difference between having $k$-rational and 
$k$-Du Bois singularities.

\begin{proof}[Proof of Corollary~\ref{krat-duality}]
By Theorem~\ref{main2}, we may assume that $Z$ has $k$-Du Bois singularities, hence $p(Z)\geq k$ by 
Theorem~\ref{thm-DB-main}.
We consider a strong log-resolution $\mu\colon \widetilde{Z} \to Z$ as in Definition \ref{def_k_rat}. We put
$W = Z_{\rm sing}$, and $D = \mu^{-1}(W)_{\rm red}$. By  \cite[Proposition~3.3]{Steenbrink}, we have an exact triangle
\begin{equation}\label{eq_exact_trg_Steenbrink}
\derR \mu_*\Omega_{\widetilde{Z}}^{d-k}({\rm log}\,D)(-D)\longrightarrow \underline{\Omega}_Z^{d-k}\longrightarrow \underline{\Omega}_W^{d-k}\overset{+1}\longrightarrow.
\end{equation}
Since $p (Z) \ge k$, using ($\ref{eqn:ineq}$) we deduce that 
$$\dim W \le d - 2k - 1 < d - k.$$
It follows that $\underline{\Omega}_W^{d-k} = 0$, hence the above triangle implies that the canonical morphism
$$\derR \mu_*\Omega_{\widetilde{Z}}^{d-k}({\rm log}\,D)(-D)\longrightarrow \underline{\Omega}_Z^{d-k}$$ 
is an isomorphism. Equivalently, using Grothendieck duality, the induced morphism
$$\derR \mathcal{H}om_{\shO_Z} ( \underline{\Omega}_Z^{d-k}, \omega_Z) \overset{\nu_k}\longrightarrow \derR \mu_*\Omega_{\widetilde{Z}}^k ({\rm log}\,D)$$
is an isomorphism.  On the other hand, since 
$\Omega_Z^k \to \underline{\Omega}_Z^k$ is an isomorphism, it follows that the canonical morphism 
$\Omega_Z^k \to \derR \mu_*\Omega_{\widetilde{Z}}^k ({\rm log}\,D)$ can be identified with the composition 
$\nu_k \circ \psi_k$. Thus by definition $Z$ has $k$-rational singularities if and only if 
$\psi_k$ is an isomorphism.
\end{proof}

Note that in the situation of Corollary \ref{krat-duality}, it follows by duality that one can also compute $\underline{\Omega}_Z^{d-k}$ in terms of the sheaf of K\"ahler differentials $\Omega_Z^k$.

\begin{corollary}\label{dual-dual}
If $Z$  has $k$-rational singularities, then we have an isomorphism
$$\psi_{d-k} \colon \underline{\Omega}_Z^{d-k} \longrightarrow \derR \mathcal{H}om_{\shO_Z} (\Omega_Z^k, \omega_Z).$$
\end{corollary}

\begin{remark}
The proof of Corollary~\ref{krat-duality} shows that if ${\rm codim}_Z(Z_{\rm sing})>k$, then
we have a natural isomorphism
$$\derR \mathcal{H}om_{\shO_Z} ( \underline{\Omega}_Z^{d-k}, \omega_Z) \simeq
\derR \mu_*\Omega_{\widetilde{Z}}^k ({\rm log}\,D).$$
In particular, this is the case if $2 p (Z) + 1 > k$.
\end{remark}

\begin{remark}[Depth of $\Omega_Z^k$]
It is well known (see e.g. \cite[Theorem 7.29]{PS}) that for any variety $W$ and any $p$, we have  $\shH^i(\underline{\Omega}_W^p) = 0$ for $i >  \dim W - p$. Thus a numerical consequence of Corollary \ref{dual-dual} is that if $Z$ has $k$-rational singularities, then
${\rm depth}(\Omega_Z^k) \ge  d- k$.
However, this holds  even if we only assume that $Z$ has $k$-Du Bois singularities:
this follows from the Auslander-Buchsbaum formula and the fact that in this case 
$C_k^\bullet$ gives a locally free resolution of $\Omega_Z^k$ (see the discussion before the proof of Theorem~\ref{injectivity}).
\end{remark}

\begin{remark}[Small singular locus]
We also note that the duality morphism $\psi_k$ has a more familiar interpretation 
for an arbitrary $d$-dimensional irreducible variety $Z$ whose singular locus $W$ has small dimension. 
Indeed, it follows from the exact triangle (\ref{eq_exact_trg_Steenbrink}) that the 
canonical morphism 
$$
\derR \mu_*\Omega_{\widetilde{Z}}^k ({\rm log}\,D)(-D)\longrightarrow \underline{\Omega}_Z^k
$$
is an isomorphism when $\dim W < k$, while just as in the proof of Corollary \ref{krat-duality}, the canonical 
morphism 
$$\derR \mathcal{H}om_{\shO_Z} ( \underline{\Omega}_Z^{d-k}, \omega_Z) \longrightarrow \derR \mu_*\Omega_{\widetilde{Z}}^k ({\rm log}\,D)$$
is an isomorphism when $\dim W < d- k$. 

We conclude that, in general, under the assumption $\dim W < {\rm min}\{k, d- k\}$, the duality morphism 
$\psi_k$ is naturally identified with the morphism 
$$\derR \mu_*\Omega_{\widetilde{Z}}^k ({\rm log}\,D)(-D)\longrightarrow  \derR \mu_*\Omega_{\widetilde{Z}}^k ({\rm log}\,D)$$
obtained by pushing forward the canonical inclusion on $\widetilde{Z}$. This holds for instance when
$Z$ has isolated singularities and $1\leq k\leq d-1$; in this case, at the level of cohomology the map can be studied as in \cite[\S3]{FL1}, using classical Hodge theory arguments involving the mixed Hodge structure on the cohomology of the link of the singularity. 
\end{remark}

\subsection{Local vanishing for $k$-Du Bois singularities}
We continue to assume that $Z \subseteq X$ is a local complete intersection closed subvariety of pure codimension $r$, and dimension $d = n-r$. For completeness, we start by recording a consequence of results we established in \cite[\S5.2]{MP2}. 
Let $f \colon Y \to X$ be a log resolution of the pair $(X, Z)$ and let $E = f^{-1} (Z)_{\rm red}$ (recall that we always assume that $f$ is an isomorphism
over $X\smallsetminus Z$).

\begin{corollary}
If $k$ is a nonnegative integer such that $p (Z) \ge k$, then 
$$R^q f_* \Omega_Y^p (\log E ) (-E) = 0 \,\,\,\,\,\,{\rm for~all} \,\,\,\,q \ge 1 \,\,\,\,{\rm and} \,\,\,\,p \le k.$$
\end{corollary}
\begin{proof}
The hypothesis implies that $\shH^q (\underline\Omega_Z^p) = 0$ in the range of the conclusion, by
Theorem \ref{thm-DB-main}. The proof is then identical to that of  \cite[Corollary 1.2]{MOPW}.
\end{proof}

We next establish the generalization of \cite[Theorem 1.4]{MOPW} to local complete intersections.

\begin{theorem}\label{marginal-LCI}
For every $q \le p (Z) + 1$ we have 
$$\shH^{d - q} (\underline\Omega_Z^q) = 0,$$
unless $q =d$ (in which case $Z$ is either smooth or a curve with nodal singularities). In particular, 
we have
\begin{equation}\label{eq1_marginal_LCI}
R^{d-q+1} f_* \Omega_Y^q (\log E ) (-E) = 0,
\end{equation}
unless $q=d$ or $q=d+1$ (in the latter case $Z$ being smooth). 
\end{theorem}

\begin{proof}
Note first that by \cite[Proposition~3.3]{Steenbrink}, we have an exact triangle
\begin{equation}\label{eq0_marginal-LCI}
\derR f_*\Omega_{Y}^q({\rm log}\,E)(-E)\longrightarrow \Omega_X^q\longrightarrow \underline{\Omega}_Z^q\overset{+1}\longrightarrow.
\end{equation}
This immediately implies that (\ref{eq1_marginal_LCI}) holds if $Z$ is smooth and $q\neq d+1$. The first assertion in the theorem
is also clear if $Z$ is smooth, hence from now on we assume that $Z$ is singular.

In this case it follows from (\ref{eqn:ineq}) that 
$p (Z) \le \frac{d-1}{2}$, hence 
\begin{equation}\label{eq2_marginal-LCI}
q \le \frac{d+1}{2}<d+1.
\end{equation}
In particular, we have $q \neq d +1$. 
We thus have $\shH^{d +1- q} (\Omega_X^q) = 0$, and the vanishing in (\ref{eq1_marginal_LCI})
follows if we know that $\shH^{d - q} (\underline\Omega_Z^q) = 0$. 

Therefore it is enough to prove the first assertion in the theorem. Using our previous constructions, this will be reduced to a statement in commutative algebra. 
 We know that $q\leq d$ and it follows from (\ref{eq2_marginal-LCI})
that if $q=d$, then $d=1$. Hence $Z$ is a locally complete intersection curve with Du Bois singularities, and it is known that 
this implies that $Z$ has nodal singularities. (Note that in this case we
clearly have $\cH^{0}(\underline{\Omega}_Z^1)\neq 0$.)

From now on we assume that $q\leq d-1$. We use the notation in the proof of Theorem~\ref{injectivity}.
By ($\ref{isom-DB}$) we have
$$\cH^{d-q}(\underline{\Omega}_Z^q)
\simeq {\mathcal Ext}^n_{\shO_X}\big({\rm Gr}_{q-n}^F{\rm DR}_X(\cH^r_Z(\shO_X)),\omega_X\big).$$
The hypothesis $p (Z) + 1\geq q$ implies that $F_i\shH^r(\shO_Z)=E_i\shH^r(\shO_Z)$ for all $i\leq q-1$, 
hence the morphism of complexes
$$\varphi_q\colon {\rm Gr}_{q-n}^F{\rm DR}_X(\cH^r_Z(\shO_X))\to B_q^{\bullet}: = {\rm Gr}_{q-n}^E{\rm DR}_X(\cH^r_Z(\shO_X))$$
is injective and its cokernel is a sheaf $\shF$ (supported in cohomological degree $0$).
We thus have an exact sequence
$${\mathcal Ext}^n_{\shO_X}(B_q^{\bullet},\omega_X)\to {\mathcal Ext}^n_{\shO_X}\big({\rm Gr}_{q-n}^F{\rm DR}_X(\cH^r_Z(\shO_X)),\omega_X\big)
\to {\mathcal Ext}^{n+1}_{\shO_X}(\shF,\omega_X)=0,$$
and we see that it is enough to show that ${\mathcal Ext}^n_{\shO_X}(B_q^{\bullet},\omega_X)=0$. 

Recall from the proof of Theorem~\ref{injectivity} that $B_q^\bullet$ can be described as the complex 
$$
0\to \Omega_X^{n-q}\otimes\omega_Z\otimes\omega_X^{-1}\to \Omega_X^{n-q+1}\otimes N_{Z/X}\otimes\omega_Z\otimes\omega_X^{-1}\to
\cdots\to \Omega_X^n\otimes {\rm Sym}^q N_{Z/X}\otimes\omega_Z\otimes\omega_X^{-1}\to 0
$$
placed in cohomological degrees $-q,\ldots,0$. The spectral sequence
$$E_1^{i,j} : = {\mathcal Ext}^j_{\shO_X}(B_q^{-i},\omega_X)\Rightarrow {\mathcal Ext}_{\shO_X}^{i+j}(B_q^{\bullet},\omega_X)$$
shows that it is enough to prove that 
$${\mathcal Ext}^{n-m}_{\shO_X} \big(\Omega_X^{n- m} \otimes  {\rm Sym}^{q-m} N_{Z/X} \otimes {\rm det} N_{Z/X},\omega_X\big) = 0\quad\text{for}\quad 0\leq m\leq q.$$
This is a consequence of the fact that $q \le d -1$, combined with the fact that if $r=n-d$, then for every locally free sheaf $\shE$ on $Z$ we have 
$$
{\mathcal Ext}_{\shO_X}^j(\shE,\omega_X)=0 \,\,\,\,\,\,{\rm for} \,\,\,\,j \neq r.
$$
\end{proof}

 We are now able to prove the local vanishing theorem for $k$-Du Bois singularities.
 
 \begin{proof}[Proof of Theorem~\ref{DB-vanishing}]
By a general vanishing result due to Steenbrink (see \cite[Theorem~2]{Steenbrink}), we always have
$R^q\mu_*\Omega^p_{\widetilde{Z}}(\log\,D) (-D)=0$ if $p+q> d$, i.e. the statement in i) (for an arbitrary variety $Z$). From now on we assume that $p+q\leq d$. 

By \cite[Proposition~3.3]{Steenbrink},  if $W=Z_{\rm sing}$, then we have an exact triangle
$$
\derR \mu_*\Omega_{\widetilde{Z}}^p({\rm log}\,D)(-D)\longrightarrow\underline{\Omega}_Z^p\longrightarrow \underline{\Omega}_W^p\overset{+1}\longrightarrow.
$$
Taking the long exact sequence in cohomology, we obtain the exact sequence
\begin{equation}\label{eq10_DB}
\cH^{q-1}(\underline{\Omega}^p_W)\longrightarrow R^q\mu_*\Omega_{\widetilde{Z}}^p(\log\,D)(-D)\longrightarrow\cH^q(\underline{\Omega}_Z^p).
\end{equation}

By Theorem \ref{thm-DB-main}, $Z$ having  $k$-Du Bois singularities is equivalent to the fact that $p(Z)\geq k$. 
This in turn implies by ($\ref{eqn:ineq}$) that $\dim(W)\leq d-2k-1$. Using a general result on the vanishing of the cohomologies of the graded pieces of the Du Bois complex (see \cite[Theorem~7.29]{PS}), we then have
$$
\cH^{q-1}(\underline{\Omega}_W^p)=0\quad\text{if}\quad p+q \ge d-2k + 1.
$$
On the other hand, if $p\leq k$ and $q\geq 1$, we have 
$\cH^q(\underline{\Omega}_Z^p)=0$ by definition of $k$-Du Bois singularities. 
The assertion in ii) thus follows from the exact sequence (\ref{eq10_DB}).

 The assertion in iii) follows directly from Theorem \ref{marginal-LCI}; note that we avoid the exceptional cases in that theorem
 since $q\leq d-1$. Let's now prove the assertion in iv). 
 
 We assume that $q\geq\max\{d-k-1,1\}$. Suppose first that $p+q\geq d-2k+1$. If $p\leq k$, then we are done by ii). Hence we may and will assume that
 $p\geq k+1$. By i), we may also assume that $p+q\leq d$, hence $q\leq d-p\leq d-k-1$. Since we are assuming $q\geq d-k-1$, it follows that
 $q=d-k-1$ and $p=k+1$, hence we are done by iii).
 
 Suppose now that $p+q\leq d-2k$. Since $q\geq d-k-1$, it follows that $p+k\leq 1$, 
 hence we are left with 3 possible case, when
$(k,p)=(1,0)$, $(0,0)$, or $(0,1)$.

If $p=0$, then we need to show that 
$$R^q\mu_*\shO_{\widetilde{Z}}(-D)=0 \,\,\,\,\,\, {\rm for} \,\,\,\, q\geq \max \{d-k-1, 1\}.$$
If $k=1$, then the condition $p+q\leq d-2k$ gives $d\geq 3$, while the inequality (\ref{eqn:ineq}) gives $\dim(W)\leq d-3$. 
If $k=0$, then we only get $\dim(W)\leq d - 2$, using the fact that $Z$ is normal. Since the case when $Z$ is a smooth curve is obvious, we assume $d \ge 2$. Using the fact that $Z$ is normal, with Du Bois singularities, we can then apply
\cite[Theorem~13.3]{GKKP} which says that $R^i\mu_*\shO_{\widetilde{Z}}(-D)=0$ for $i>\max\{\dim(W), 0\}$. Therefore we are done in this case.

If $p=1$ and $k=0$, then we have $d \ge 2$ and we need to show that 
$$R^q\mu_*\Omega_{\widetilde{Z}}^1(\log\,D)(-D)=0 \,\,\,\,\,\,{\rm  for}\,\,\,\, q\geq d-1.$$  
This follows by applying \cite[Theorem~14.1]{GKKP}.\footnote{Note that the result we refer to is stated for log canonical pairs $(Z,E)$. However, when $E=0$,  the same proof works 
when $Z$ is only assumed to have Du Bois singularities; this is the only condition that is used in the proof, via \cite[Theorem~13.3]{GKKP}.}
This completes the proof of the theorem.
\end{proof}

\section{Results for hypersurfaces}
In this chapter, we show that $k$-rational singularities are characterized numerically in terms of their
minimal exponent; see the discussion around Theorem \ref{thm_main} and Corollary \ref{drop} in the introduction.

All throughout this chapter, $Z$ will be a (nonempty) hypersurface, meaning a closed subscheme of pure codimension $1$, in the smooth irreducible $n$-dimensional variety $X$.

\subsection{Minimal exponent $>k+1$ implies $k$-rational singularities}
 Our goal in this section is to prove one implication in Theorem~\ref{thm_main}, namely to show the following:

\begin{theorem}\label{main_thm1}
If $k$ is a nonnegative integer such that 
$\widetilde{\alpha}(Z)>k+1$, then $Z$ has $k$-rational singularities.
\end{theorem}

\begin{remark}\label{rmk1_main_thm1}
Under the hypotheses in Theorem~\ref{main_thm1}, we have $\widetilde{\alpha}(Z)>1$, hence $Z$ is normal; in particular, it is reduced 
and all its connected components are irreducible.\footnote{It follows from \cite[Theorem~0.4]{Saito-B} that $\widetilde{\alpha}(Z)>1$ implies that $Z$ has rational singularities, hence in particular it is normal. 
However, we prefer to give a direct argument for this elementary fact.}
Indeed, we have ${\rm lct}(X,Z)=1$, hence $Z$ is reduced and in codimension $1$ its singularities are at worst nodal singularities. Since the Bernstein-Sato polynomial
of a union of two smooth divisors meeting transversally is
$(s+1)^2$, we conclude that if $\widetilde{\alpha}(Z)>1$, we have $Z$ smooth in codimension $1$. Therefore $Z$ is normal by Serre's criterion since it is Cohen-Macaulay,
being a hypersurface. 
\end{remark}

In order to prove the theorem, we may clearly treat each connected component of $Z$ separately. In light of Remark~\ref{rmk1_main_thm1}, we thus may and will
assume from now on that $Z$ is normal. We choose a log resolution $\pi\colon Y\to X$ of $(X,Z)$
and write $\pi^{-1}(Z)_{\rm red}=E$. The morphism $\pi$ may be taken to be projective, and we will do so in what follows. The key input for the proof of Theorem~\ref{main_thm1} is the following vanishing result, interesting in its own right.

\begin{theorem}\label{thm_vanishing_log}
If $k$ is a nonnegative integer and $\widetilde{\alpha}(Z)>k+1$, then 
$$R^q\pi_*\Omega_Y^p(\log\,E)=0$$ 
in each of the following cases:
\begin{enumerate}
\item[i)] $p+q\geq n+1$;
\item[ii)] $p\leq k+1$ and $q\geq 1$;
\item[iii)] $p=k+2$ and $q=n-k-2$. 
\end{enumerate}
In particular, it holds for all $p$ and all $q \ge \max \{ n- k - 2, 1\}$.
\end{theorem}

Part i) is known to hold for an arbitrary hypersurface $Z$ in $X$; cf. ($\ref{eq_vanishing_sum_n}$) below.
The case when $p=k+2$ and $q=n-k-2$ is the content of \cite[Corollary~C]{MP1}. 
In what follows we will show that (a simplified version of) the argument in \emph{loc}.~\emph{cit}. is enough to give the vanishing in ii) as well.

\begin{remark}\label{rmk_thm_vanishing}
We note that if $Z$ is smooth, then all the vanishings in Theorem~\ref{thm_vanishing_log} are well known; see for example \cite[Theorem~31.1(i)]{MP0}.
On the other hand, if $Z$ is singular, we have $\widetilde{\alpha}(Z)\leq\tfrac{n}{2}$ by \cite[Theorem~0.4]{Saito_microlocal}. If $\widetilde{\alpha}(Z)>k+1$,
then $k\leq \tfrac{n-3}{2}$, 
hence if $q\geq n-k-2$, we automatically have $q\geq 1$. 
\end{remark} 

We begin by relating the higher direct images of the sheaves of log differentials on $Y$ to certain mixed Hodge modules on $X$.
Recall that $\omega_X(*Z)$ is the right Hodge module whose underlying $\Dmod_X$-module is the module of top rational differentials on $X$
with poles along $Z$. It is a basic fact that for every $p$ we have an isomorphism
\begin{equation}\label{eq_push_forward1}
\derR \pi_*\Omega_Y^p(\log\,E)[n-p]\simeq {\rm Gr}^F_{-p}{\rm DR}_X\big(\omega_X(*Z)\big).
\end{equation}
One way to see this is by using the explicit filtered resolution of $\omega_Y(*E)$ in \cite[Proposition~3.1]{MP0} to get an
isomorphism
$$\Omega_Y^p(\log\,E)[n-p]\simeq {\rm Gr}^F_{-p}{\rm DR}_Y\big(\omega_Y(*E)\big).$$
The isomorphism in (\ref{eq_push_forward1}) then follows from the fact that $\omega_X(*Z)$ is the push-forward of $\omega_Y(*E)$ (in the category of 
Hodge modules), using Saito's Strictness Theorem \cite[Section~2.3.7]{Saito-MHP}; cf.\ \cite[Section C.4]{MP0}. 

Since the complex on the right-hand side of 
(\ref{eq_push_forward1}) is placed in nonpositive degrees, we immediately deduce 
the following vanishing result
\begin{equation}\label{eq_vanishing_sum_n}
R^q\pi_*\Omega_Y^p(\log\,E)=0\quad\text{for}\quad p+q>n
\end{equation}
(cf. \cite[Corollary~3]{Saito-LOG} and \cite[Theorem~32.1]{MP0}). In other words, part i) in Theorem
\ref{thm_vanishing_log} holds for an arbitrary hypersurface $Z$ in $X$.

\begin{lemma}\label{van_for_loc_coh}
With the notation above, we have $R^q\pi_*\Omega_Y^p(\log\,E)=0$ for $q \ge 1$ and arbitrary $p$,
provided that 
$$\cH^{p+q-n}\big({\rm Gr}^F_{-p}{\rm DR}_X\big(\cH^1_Z(\omega_X)\big)\big)=0.$$
\end{lemma}

\begin{proof}
We have a short exact sequence of filtered right $\Dmod_X$-modules underlying mixed Hodge modules
$$0\longrightarrow \omega_X \longrightarrow \omega_X(*Z)\longrightarrow \cH_Z^1(\omega_X)\longrightarrow 0,$$
where the underlying $\Dmod_X$-module of $\cH_Z^1(\omega_X)$ is the first local cohomology module of $\omega_X$ along $Z$.
By taking ${\rm Gr}^F_{-p}{\rm DR}_X(-)$,  using (\ref{eq_push_forward1}) we obtain an exact triangle
$$\Omega_X^p[n-p]\longrightarrow \derR \pi_*\Omega_Y^p(\log\,E)[n-p]\longrightarrow {\rm Gr}^F_{-p}{\rm DR}_X\big(\cH^1_Z(\omega_X)\big)\overset{+1}\longrightarrow.$$
This immediately implies the statement by passing to cohomology.
\end{proof}

The key point for the proof of Theorem \ref{thm_vanishing_log} is the use of duality. Recall that by the compatibility between duality and the graded de Rham complex (see \cite[Sections~2.4.5 and 2.4.11]{Saito-MHP}),
for every $p\in\ZZ$, we have an isomorphism 
$$
{\rm Gr}^F_{-p}{\rm DR}_X\big(\cH^1_Z(\omega_X)\big)\simeq \derR {\mathcal Hom}_{\shO_X}\big({\rm Gr}^F_{p}{\rm DR}_X{\mathbf D}(\cH^1_Z(\omega_X)),\omega_X[n]\big),
$$
and thus 
\begin{equation}\label{eq_isom_dual}
\cH^{p+q-n}\big({\rm Gr}^F_{-p}{\rm DR}_X\big(\cH^1_Z(\omega_X)\big)\big)\simeq {\mathcal Ext}_{\shO_X}^{p+q}\big({\rm Gr}_p^F{\rm DR}_X{\mathbf D}(\cH^1_Z(\omega_X)),\omega_X\big).
\end{equation}

We compute ${\mathbf D}(\cH^1_Z(\omega_X))$ using the $V$-filtration. We work locally and suppose that $Z$ is a singular hypersurface in $X$ defined by $f\in\shO_X(X)$. 
Recall that if $\iota\colon X\to X\times\CC$ is the graph embedding $\iota(x)=\big(x,f(x)\big)$, then the $V$-filtration of Malgrange and Kashiwara 
is an increasing, exhaustive, discrete, and right-continuous filtration on 
$$B_f:=\iota_+\omega_X=\bigoplus_{j\geq 0}\omega_X\partial_t^j$$ 
indexed by rational numbers and characterized by a few conditions
(see for example \cite[Section~3.1]{Saito-MHP}). For example, if $t$ is the coordinate on $\CC$, then $V_{\alpha}B_f\cdot t\subseteq V_{\alpha-1}B_f$
and $V_{\alpha}B_f\cdot\partial_t\subseteq V_{\alpha+1}B_f$. 
For every $\alpha \in \QQ$, we put ${\rm Gr}^V_{\alpha}B_f=V_{\alpha}B_f/V_{<\alpha}B_f$. 
We note that the Hodge filtration on $B_f$ is given by
$$F_{p-n}B_f=\bigoplus_{j\leq p}\omega_X\partial_t^j.$$

By \cite[Section~2.24]{Saito-MHM},
we have a short exact sequence
$$0\to {\rm Gr}_0^VB_f\overset{\cdot t}\longrightarrow{\rm Gr}^V_{-1}B_f\to \cH^1_Z(\omega_X)\to 0,$$
 which after taking duals gives the exact sequence
\begin{equation}\label{eq_duality1}
0\to {\mathbf D}\big(\cH^1_Z(\omega_X)\big)\to {\mathbf D}({\rm Gr}^V_{-1}B_f)\to {\mathbf D}({\rm Gr}^V_{0}B_f)\to 0.
\end{equation}
On the other hand, by  \cite[Theorem~1.6]{Saito_duality} we have isomorphisms
$${\mathbf D}({\rm Gr}_{-1}^VB_f)\simeq {\rm Gr}_{-1}^VB_f(n+1)\quad\text{and}\quad {\mathbf D}({\rm Gr}_0^VB_f)\simeq {\rm Gr}^V_0B_f(n)$$
such that the exact sequence (\ref{eq_duality1}) becomes
\begin{equation}\label{eq_duality2}
0\to {\mathbf D}\big(\cH^1_Z(\omega_X)\big)\to {\rm Gr}^V_{-1}B_f(1+n)\overset{\cdot(-\partial_t)}\longrightarrow {\rm Gr}^V_{0}B_f(n)\to 0.
\end{equation}
By taking the corresponding graded piece of the de Rham complex, we obtain the short exact sequence of complexes
\begin{equation}\label{eq_duality3}
0\to {\rm Gr}^F_p{\rm DR}_X{\mathbf D}\big(\cH^1_Z(\omega_X)\big)\to {\rm Gr}^F_{p-n-1}{\rm DR}_X{\rm Gr}^V_{-1}B_f\overset{\cdot\partial_t}\longrightarrow
{\rm Gr}^F_{p-n}{\rm DR}_X{\rm Gr}^V_{0}B_f\to 0.
\end{equation}

The connection between the minimal exponent and the $V$-filtration is provided by the following result due to Saito, see \cite[(1.3.8)]{Saito-MLCT}:
if $q$ is a nonnegative integer and $\alpha\in (0,1]$ is a rational number, then 
\begin{equation}\label{eq_char_min_exp}
\widetilde{\alpha}(Z)\geq q+\alpha \iff \omega_X\partial_t^q\subseteq V_{-\alpha}B_f \,\,(\iff F_{q-n}B_f \subseteq V_{-\alpha}B_f).
\end{equation}
(Note that the formulation in \emph{loc}.~\emph{cit}. is in terms of the left $\Dmod_X$-module $B_f^{\ell}$ corresponding to our $B_f$, and with respect to the decreasing version of the 
$V$-filtration, related to the one in this paper by $V_{\alpha}B_f=\omega_X\otimes V^{-\alpha}B_f^{\ell}$.) 
Therefore the condition that $\widetilde{\alpha}(Z)>k+1$ is equivalent to $F_{k+1-n}B_f\subseteq V_{<0}B_f$. Note that this implies 
\begin{equation}\label{eq5_char_min_exp}
{\rm Gr}^F_{i-n}{\rm Gr}^V_0B_f=0\quad\text{for all}\quad  i\leq k+1.
\end{equation}
Moreover, it was shown in \cite[Proposition~4.5]{MP1} that (after translating from left to right $\Dmod_X$-modules) 
the condition $\widetilde{\alpha}(Z)>k+1$ implies
\begin{equation}\label{eq2_char_min_exp}
{\rm Gr}^F_{i-n}{\rm Gr}^V_{-1}B_f\simeq \omega_X \otimes_{\shO_X} \shO_X/(f)\quad\text{for}\quad 0\leq i\leq k\quad\text{and}
\end{equation}
\begin{equation}\label{eq3_char_min_exp}
{\rm Gr}^F_{k+1-n}{\rm Gr}^V_{-1}B_f\simeq \omega_X \otimes_{\shO_X} J/(f),
\end{equation}
where $J=\{h\in\shO_X\mid \omega_X\cdot h\partial_t^{k+1}\subseteq V_{-1}B_f\}$. We also note that
\begin{equation}\label{eq4_char_min_exp}
{\rm Gr}^F_{i-n}{\rm Gr}^V_{-1}B_f=0\quad\text{for}\quad i<0
\end{equation}
since $F_{i-n}B_f=0$ for $i<0$. 

After these preparations we can give the proof of our vanishing result.

\begin{proof}[Proof of Theorem~\ref{thm_vanishing_log}]
The vanishing in i) follows from (\ref{eq_vanishing_sum_n}). Since all assertions are local on $X$, to approach the rest we may assume that $Z$ is defined by $f\in\shO_X(X)$. We may also assume that
$Z$ is singular; see Remark~\ref{rmk_thm_vanishing}.

We begin by recalling that if $A^{\bullet}$ is a complex of $\shO_X$-modules, then we have a spectral sequence
$$E_1^{i,j}={\mathcal Ext}_{\shO_X}^j(A^{-i},\omega_X)\Rightarrow {\mathcal Ext}_{\shO_X}^{i+j}(A^{\bullet},\omega_X).$$
We thus obtain for every $m\in\ZZ$
\begin{equation}\label{eq1_pf_thm_vanishing}
{\mathcal Ext}^m_{\shO_X}(A^{\bullet},\omega_X)=0\quad\text{if}\,\,\,\,\,\,{\mathcal Ext}_{\shO_X}^{m+i}(A^i,\omega_X)=0\,\,\text{for all}\,\,i\in\ZZ.
\end{equation}

We also recall that by Lemma~\ref{van_for_loc_coh}, in order to prove that
$R^q\pi_*\Omega_X^p(\log\,E)=0$ for some $p$ and some $q \ge 1$, it is enough to verify that
$$\cH^{p+q-n}\big({\rm Gr}^F_{-p}{\rm DR}_X(\cH_Z^1(\omega_X))\big)=0.$$
Furthermore, the isomorphism (\ref{eq_isom_dual}) implies that if
$$\Mmod_{p,q}:={\mathcal Ext}_{\shO_X}^{p+q}\big({\rm Gr}^F_p{\rm DR}_X{\mathbf D}(\cH_Z^1(\omega_Z)),\omega_X\big),$$
then it is enough to show that $\Mmod_{p,q}=0$. 

Let's now prove  ii). Suppose that $p\leq k+1$ and $q\geq 1$. 
The long exact sequence associated to the exact sequence of complexes (\ref{eq_duality3}) gives an exact sequence
\begin{equation}\label{eq2_pf_thm_vanishing}
{\mathcal Ext}^{p+q}_{\shO_X}\big({\rm Gr}^F_{p-n-1}{\rm DR}_X{\rm Gr}^V_{-1}B_f,\omega_X\big)\to\Mmod_{p,q}\to {\mathcal Ext}^{p+q+1}_{\shO_X}\big({\rm Gr}^F_{p-n}{\rm DR}_X{\rm Gr}^V_0B_f,\omega_X\big).
\end{equation}
Since ${\rm Gr}^F_{i-n}{\rm Gr}_0^VB_f=0$ for all $i\leq p\leq k+1$ by (\ref{eq5_char_min_exp}), it follows that 
$${\rm Gr}^F_{p-n}{\rm DR}_X{\rm Gr}^V_0B_f=0,$$ 
hence the third term
in (\ref{eq2_pf_thm_vanishing}) is $0$. On the other hand, if $A^{\bullet}={\rm Gr}^F_{p-n-1}{\rm DR}_X{\rm Gr}^V_{-1}B_f$, then $A^i=0$ unless 
$1-p\leq i\leq 0$. Since each $A^i$ is a locally free $\shO_Z$-module by (\ref{eq2_char_min_exp}), we have ${\mathcal Ext}^j_{\shO_X}(A^i,\omega_X)=0$ for all $j\geq 2$.
But for $i\geq 1-p$ we have $p+q+i\geq q+1\geq 2$, hence it follows that
$${\mathcal Ext}_{\shO_X}^{p+q+i}(A^i,\omega_X)=0 \,\,\,\,\,\,{\rm  for ~all} \,\,\,\, i\in\ZZ.$$ 
We thus conclude, using (\ref{eq1_pf_thm_vanishing}), that the first term in (\ref{eq2_pf_thm_vanishing}) is $0$ as well. Therefore $\Mmod_{p,q}=0$, which completes the proof of ii).

Suppose now that $p=k+2$ and $q=n-k-2\geq 1$ and let's prove that again $\Mmod_{p,q}=0$. 
The long exact sequence associated with the exact sequence of complexes (\ref{eq_duality3}) gives an exact sequence
\begin{equation}\label{eq_3_pf_thm_vanishing}
{\mathcal Ext}_{\shO_X}^n\big({\rm Gr}^F_{k-n+1}{\rm DR}_X{\rm Gr}^V_{-1}B_f,\omega_X\big)\to\Mmod_{p,q}\to 
 {\mathcal Ext}^{n+1}_{\shO_X}\big({\rm Gr}^F_{k-n+2}{\rm DR}_X{\rm Gr}^V_0B_f,\omega_X\big).
\end{equation}
Note first that (\ref{eq5_char_min_exp}) implies that ${\rm Gr}^F_{k-n+2}{\rm DR}_X{\rm Gr}^V_0B_f$ is concentrated in cohomological degree $0$.
Since $\dim(X)=n$, it follows that the third term in (\ref{eq_3_pf_thm_vanishing}) is 0. 

We now consider the complex $A^{\bullet}={\rm Gr}^F_{k-n+1}{\rm DR}_X{\rm Gr}^V_{-1}B_f$. We have $A^i=0$ unless $-k-1\leq i\leq 0$ and 
it follows from (\ref{eq2_char_min_exp}) that $A^i$ is a free $\shO_Z$-module for $i\leq -1$, while $A^0\simeq 
\omega_X\otimes_{\shO_X} J/(f)$. 
For $-k-1\leq i\leq -1$, we thus have $n+i\geq n-k-1\geq 2$, hence ${\mathcal Ext}_{\shO_X}^{n+i}(A^i,\omega_X)=0$. 
Moreover, it follows from (\ref{eq2_char_min_exp}) that we have an exact sequence
$$0\to A^0 \otimes_{\shO_X} \omega_X^{-1} \to\shO_Z\to\shO_X/J\to 0,$$
which immediately gives ${\mathcal Ext}^n_{\shO_X}(A^0,\omega_X)=0$. Using 
(\ref{eq1_pf_thm_vanishing}) one more time, we conclude that the first term in (\ref{eq_3_pf_thm_vanishing}) is also $0$, and thus $\Mmod_{p,q}=0$. This completes the proof of the theorem. 
\end{proof}

We next deduce the local vanishing result for a resolution of $Z$ stated in the introduction.

\begin{proof}[Proof of Theorem~\ref{thm_vanishing}]
Note first that by Lemma~\ref{indep_res2}, the result is independent of the choice of strong log resolution $\mu$. 
Let $\pi\colon Y\to X$ be a strong log resolution of $(X,Z)$ as in Proposition~\ref{prop_log_res} and 
$\mu\colon\widetilde{Z}\to Z$ the induced log resolution of $Z$. Recall that we write 
$$\pi^{-1}(Z)_{\rm red}=E=\widetilde{Z}+F_N \,\,\,\,\,\,{\rm  and} \,\,\,\,\,\, D=F_N\vert_{\widetilde{Z}}=\mu^{-1}(Z_{\rm sing})_{\rm red}.$$

Note that since $\widetilde{\alpha}(Z)>k+1$, it follows from \cite[Proposition~7.4]{MP1} that $r:={\rm codim}_X(Z_{\rm sing})\geq 2k+3$. Thus 
(\ref{eq11_prop_log_res})  in Proposition \ref{prop_log_res} gives 
$$R^q\mu_*\Omega_{\widetilde{Z}}^p(\log D)\simeq R^q\pi_*\Omega_Y^{p+1}(\log\,E)$$
for all $q \ge 1$ and all $0 \le p \le 2k + 1$. This concludes the proof in combination with assertions ii) and iii) 
in Theorem \ref{thm_vanishing_log}.
\end{proof}

\begin{remark}[Isolated singularities]
Under the hypothesis of Theorem \ref{thm_vanishing}, if $Z$ has at most isolated singularities, then for $q\geq \max \{n-k-2, 1\}$ we have
$$R^q\mu_*\Omega_{\widetilde{Z}}^p(\log\,D)=0 \iff (p,q)\neq (n-1,n-2).$$

Indeed, by Remark~\ref{rmk_indep_res} the assertion is independent of the choice of log resolution. Hence we may assume that we are in the setting of Proposition~\ref{prop_log_res}, which implies
$$R^q\mu_*\Omega_{\widetilde{Z}}^p(\log\,D)\simeq R^q\pi_*\Omega_Y^{p+1}(\log\,E)\quad\text{for}\quad p\leq n-2.$$
By Theorem~\ref{thm_vanishing_log}, the right-hand side vanishes for $q\geq \max\{n-k-2,1\}$.

On the other hand, if $p=n-1$, then $\Omega_{\widetilde{Z}}^p(\log\,D)\simeq \omega_{\widetilde{Z}}(D)$ and the short exact sequence
$$0\to \omega_{\widetilde{Z}}\to\omega_{\widetilde{Z}}(D)\to\omega_D\to 0$$
together with Grauert-Riemenschneider vanishing gives
$$R^q\mu_*\omega_{\widetilde{Z}}(D)\simeq R^q\mu_*\omega_D.$$
The right-hand side is a skyscraper sheaf  of length $h^q(D,\omega_D)=h^{n-2-q}(D,\shO_D)$.
We clearly have $h^0(D,\shO_D)\neq 0$ (in fact, we have $h^0(D,\shO_D)=1$ since $Z$ being normal implies that $D$ is connected), while 
$h^i(D,\shO_D)=0$ for $i>0$ since $Z$ has isolated rational singularities; see \cite[Lemma~1.2]{Namikawa}. This proves our assertion.
\end{remark}

The fact that hypersurfaces with minimal exponent $>k+1$ have $k$-rational singularities is an easy consequence.  

\begin{proof}[Proof of Theorem~\ref{main_thm1}]
By Theorem~\ref{thm_vanishing} i) we have 
$$R^q\mu_*\Omega_{\widetilde{Z}}^p(\log\,D)=0 \,\,\,\,\,\,{\rm  for ~all} \,\,\,\, q\geq 1 \,\,\,\,{\rm and} \,\,\,\, 
p\leq k.$$ 

Therefore in order to establish that $Z$ has $k$-rational singularities it is enough to show that
the canonical morphism
\begin{equation}\label{eq_main_thm1}
\Omega_Z^i\to \mu_*\Omega_{\widetilde{Z}}^i(\log\,D)
\end{equation}
is an isomorphism for $i\leq k$. This is clear if $k=0$ by Zariski's Main Theorem, since $Z$ is normal; see Remark~\ref{rmk1_main_thm1}.

Suppose now that $k\geq 1$. Since we already know that $Z$ has rational singularities, the morphism (\ref{eq_main_thm1}) is an isomorphism if and only if
$\Omega_Z^i$ is reflexive; see Remark~\ref{rmk_level_zero} and Lemma \ref{lem:reflexive}. This in turn follows for $i\leq k$ from the fact that $\widetilde{\alpha}(Z)\geq k+1$, according to \cite[Remark~3.5]{MOPW}, and completes the proof of the theorem. 
\end{proof}

\subsection{$k$-rational hypersurface singularities have minimal exponent $>k+1$}
This section is devoted to the proof of the converse statement in Theorem \ref{thm_main}.

\begin{theorem}\label{converse}
If $k$ is a nonnegative integer and $Z$ has $k$-rational singularities, then $\widetilde{\alpha}(Z)>k+1$.
\end{theorem}

\begin{proof}
Since $Z$ has $k$-rational singularities, it follows from 
Theorem~\ref{main2} that it has $k$-Du Bois singularities, in which case we get $\widetilde{\alpha}(Z)\geq k+1$ by \cite[Theorem~1]{Saito_et_al}.
Using this, we can argue as in the proof of Theorem~\ref{thm_vanishing_log} to show that in fact 
$\widetilde{\alpha}(Z)>k+1$.
Since the assertion we want to prove is local on $X$, we may and will assume that $Z$ is defined by some $f\in\shO_X(X)$. We also assume that 
$Z$ is singular, since otherwise the assertion is trivial. We use the notation from the previous section.

\noindent
\emph{Claim.} 
To prove the theorem, it is enough to show that 
\begin{equation}\label{eq0_converse}
{\rm Gr}^F_{k-n+1}{\rm DR}_X{\rm Gr}^V_0B_f=0
\end{equation}
in the derived category of coherent sheaves on $X$. To this end, we record two facts that use  (\ref{eq_char_min_exp}): first, the known inequality $\widetilde{\alpha}(Z)\geq k+1$ implies that $F_{k-n}B_f\subseteq V_{-1}B_f$; after applying $\partial_t$ this 
also implies $F_{k-n+1}B_f\subseteq V_0B_f$. Second, proving that 
$\widetilde{\alpha}(Z)>k+1$ is equivalent to showing that $F_{k-n+1}B_f\subseteq V_{<0}B_f$. 

Since we already know that $F_{k-n}B_f\subseteq V_{<0}B_f$
and $F_{k-n+1}B_f\subseteq V_0B_f$, it is in turn enough to show that ${\rm Gr}^F_{k-n+1}{\rm Gr}^V_0B_f=0$. Finally, this holds if and only if we have the identity in ($\ref{eq0_converse}$): indeed, since ${\rm Gr}^F_i{\rm Gr}^V_0B_f=0$ for all $i\leq k-n$, the graded piece of de Rham complex we are interested in 
is concentrated in cohomological degree $0$, where the entry is ${\rm Gr}^F_{k-n+1}{\rm Gr}^V_0B_f$. This concludes the proof of the Claim.

Recall also that by (\ref{eq_duality3}), we have a short exact sequence of complexes
$$0\longrightarrow {\rm Gr}^F_{k+1}{\rm DR}_X{\mathbf D}\big(\cH^1_Z(\omega_X)\big)\longrightarrow {\mathcal Q}_1\longrightarrow {\mathcal Q}_2\longrightarrow 0,$$
where 
$${\mathcal Q}_1:={\rm Gr}^F_{k-n}{\rm DR}_X{\rm Gr}^V_{-1}B_f\quad\text{and}\quad {\mathcal Q}_2:={\rm Gr}^F_{k-n+1}{\rm DR}_X{\rm Gr}^V_{0}B_f.$$
Since $\derR {\mathcal Hom}_{\shO_X}\big(-,\omega_X[n]\big)$ is a duality, it follows that (\ref{eq0_converse}) holds if and only if the induced morphism
$$\tau\colon \derR {\mathcal Hom}_{\shO_X}\big({\mathcal Q}_1, \omega_X[n]\big)\longrightarrow 
{\mathcal P}:=\derR {\mathcal Hom}_{\shO_X}\big({\rm Gr}^F_{k+1}{\rm DR}_X{\mathbf D}(\cH^1_Z(\omega_X)),\omega_X[n]\big)$$
is an isomorphism. 

We now describe the domain and the target of $\tau$, starting with the target. Note first that by the compatibility between duality and taking the graded pieces of the de Rham complex 
(see  \cite[Sections~2.4.5 and 2.4.11]{Saito-MHP}), we have
$${\mathcal P}\simeq {\rm Gr}^F_{-k-1}{\rm DR}_X\big(\cH^1_Z(\omega_X)\big).$$
We analyze ${\mathcal P}$ more carefully under the $k$-rationality assumption. Let $\pi\colon Y\to X$ be a strong log resolution of $(X,Z)$ as in Proposition~\ref{prop_log_res} and let $\mu\colon\widetilde{Z}\to Z$ be the induced strong log resolution of $Z$. As in the proof of Lemma~\ref{van_for_loc_coh}, we have an exact triangle
$$\Omega_X^{k+1}[n-k-1]\longrightarrow \derR \pi_*\Omega_Y^{k+1}(\log\,E)[n-k-1]
\longrightarrow \mathcal{P} \overset{+1}\longrightarrow.$$
Since $\widetilde{\alpha}(Z)\geq k+1$, it follows from \cite[Lemma~2.1]{MOPW} that ${\rm codim}_X(Z_{\rm sing})\geq 2k+2$. We can thus apply
Proposition~\ref{prop_log_res} and the fact that $Z$ has $k$-rational singularities to deduce that 
$$
R^q\pi_*\Omega_Y^{k+1}(\log\,E)=0\quad\text{for all}\quad q\geq 1\quad\text{and}
$$
$$
\Omega_Z^k\simeq {\rm Coker}\big(\Omega_X^{k+1}\hookrightarrow \pi_*\Omega_Y^{k+1}(\log\,E)\big).
$$
We thus conclude that $\cH^i({\mathcal P})=0$ for all $i\neq k+1-n$ and
$$\cH^{k+1-n}({\mathcal P})\simeq \Omega_Z^k.$$

We next describe the domain of $\tau$. Using the notation in the proof of Theorem \ref{thm_vanishing_log}, we note that the condition $\widetilde{\alpha}(Z)\geq k+1$, being equivalent to $\omega_X \cdot \partial_t^k \subseteq V_{-1} B_f$, can also be interpreted as saying that $J = \shO_X$, where 
$$J=\{h\in\shO_X\mid \omega_X\cdot h\partial_t^k\subseteq V_{-1}B_f\}.$$
Applying the formulas ($\ref{eq2_char_min_exp}$) and ($\ref{eq3_char_min_exp}$) with $k+1$ replaced by $k$, 
we obtain isomorphisms 
$${\rm Gr}_{i-n}^F{\rm Gr}^V_{-1}B_f\simeq \omega_X\otimes_{\shO_X}\shO_X/(f)\quad\text{for all}\quad 
i\leq k.$$
Moreover, it is easy to check that via these isomorphisms the complex 
${\mathcal Q}_1$ becomes isomorphic to ${\mathcal A}^{\bullet}\otimes_{\shO_X}\shO_Z$, where
${\mathcal A}^{\bullet}$ is the complex
$$0\longrightarrow \Omega_X^{n-k}\longrightarrow\cdots\longrightarrow \Omega_X^{n-1}\longrightarrow \omega_X\longrightarrow 0$$
placed in cohomological degrees $-k,\ldots,0$, with each map given by wedging with $df$.
Since ${\mathcal A}^{\bullet}$ is a complex of locally free $\shO_X$-modules, it is well known 
(and easy to show) that we have an isomorphism
$$\derR {\mathcal Hom}_{\shO_X}\big({\mathcal Q}_1,\omega_X[n]\big)\simeq ({\mathcal A}^{\bullet})^{\vee}\otimes_{\shO_X}\omega_Z[n-1].$$

Finally, we record one last consequence of the inequality $\widetilde{\alpha}(Z)\geq k+1$ needed here. 
Namely, under this assumption the sheaf $\Omega_Z^k$ is reflexive;
see \cite[Remark~3.5]{MOPW}. (Note also that $Z$ is normal since it has rational singularities.) 
Furthermore, using the isomorphism $\shO_Z (Z)\simeq\shO_Z$ provided by the global equation $f$ defining $Z$, we know that  $\Omega_Z^k$ has a locally free resolution
given by the following complex ${\mathcal C}^{\bullet}$, placed in degrees $-k,\ldots,0$:
$${\mathcal C}^{\bullet}:\,\,\,\,0\longrightarrow \shO_Z\longrightarrow \Omega_X^1\vert_Z\longrightarrow \cdots\longrightarrow\Omega_X^k\vert_Z\longrightarrow 0,$$
where all maps are given by wedging with $df$; as explained in the proof of \cite[Theorem~1]{MOPW}, this follows from the criterion for the exactness of the Koszul complex in
\cite[Theorem~16.8]{Matsumura}.
 Therefore we have 
$$\derR {\mathcal Hom}_{\shO_X}\big({\mathcal Q}_1,\omega_X[n]\big)\simeq {\mathcal C}^{\bullet}[n-1-k]\simeq\Omega_Z^k[n-1-k].$$

Putting everything together, we see that $\tau$ can be identified with a morphism 
$$\Omega_Z^k[n-1-k]\longrightarrow \Omega_Z^k[n-1-k].$$ 
Since $\Omega_Z^k$ is a reflexive sheaf it is thus enough to check
that $\tau$ is an isomorphism on $X\smallsetminus Z_{\rm sing}$. However, this is clear since all our constructions are compatible with restriction to open subsets, and 
we clearly have ${\mathcal Q}_2\vert_{X\smallsetminus Z_{\rm sing}}=0$. This completes the proof of the theorem.
\end{proof}


\section*{References}
\begin{biblist}

\bib{Bruns}{book}{
   author={Bruns, W.},
   author={Vetter, U.},
   title={Determinantal rings},
   series={Lecture Notes in Mathematics},
   volume={1327},
   publisher={Springer-Verlag, Berlin},
   date={1988},
   pages={viii+236},
}

\bib{Deligne}{article}{
   author={Deligne, P.},
   title={Th\'{e}orie de Hodge. III},
   journal={Inst. Hautes \'{E}tudes Sci. Publ. Math.},
   number={44},
   date={1974},
   pages={5--77},
}

\bib{GNPP}{book}{
   author={Guill\'en, F.},
   author={Navarro Aznar, V.},
   author={Pascual-Gainza, P.},
   author={Puerta, F.},
   title={Hyperr\'esolutions cubiques et descente cohomologique},
   journal={Springer Lect. Notes in Math.},
   number={1335},
   date={1988},
}


\bib{Huber}{article}{
   author={Huber, A.},
   author={J\"{o}rder, C.},
   title={Differential forms in the h-topology},
   journal={Algebr. Geom.},
   volume={1},
   date={2014},
   pages={449--478},
}

\bib{FL1}{article}{
author={Friedman, R.},
author={Laza, R.},
title={Deformations of singular Fano and Calabi-Yau varieties},
journal={preprint arXiv:2203.04823},
date={2022},
}

\bib{FL2}{article}{
author={Friedman, R.},
author={Laza, R.},
title={Higher Du Bois and higher rational singularities},
journal={preprint arXiv:2205.04729, to appear in Duke Math. J.},
date={2022},
}

\bib{FL3}{article}{
author={Friedman, R.},
author={Laza, R.},
title={The higher Du Bois and higher rational properties for isolated singularities},
journal={preprint arXiv:2207.07566, to appear in J. Algebraic Geom.},
date={2022},
}


\bib{GKKP}{article}{
   author={Greb, D.},
   author={Kebekus, S.},
   author={Kov\'{a}cs, S.},
   author={Peternell, T.},
   title={Differential forms on log canonical spaces},
   journal={Publ. Math. Inst. Hautes \'{E}tudes Sci.},
   number={114},
   date={2011},
   pages={87--169},
}

\bib{Hartshorne}{article}{
   author={Hartshorne, R.},
   title={Stable reflexive sheaves},
   journal={Math. Ann.},
   volume={254},
   date={1980},
   number={2},
   pages={121--176},
   }

\bib{Saito_et_al}{article}{
   author={Jung, S.-J.},
   author={Kim, I.-K.},
   author={Saito, M.},
   author={Yoon, Y.},
   title={Higher Du Bois singularities of hypersurfaces},
   journal={Proc. Lond. Math. Soc. (3)},
   volume={125},
   date={2022},
   number={3},
   pages={543--567},
   }


\bib{KS}{article}{
   author={Kebekus, S.},
   author={Schnell, C.},
   title={Extending holomorphic forms from the regular locus of a complex space to a resolution of singularities},
   journal={J. Amer. Math. Soc.},
   volume={34},
   date={2021},
   pages={315--368}
}

\bib{KL}{article}{
   author={Kerr, M.},
   author={Laza, R.},
   title={Hodge theory of degenerations, (II): vanishing cohomology and geometric applications},
   journal={preprint arXiv:2006.03953, to appear in Current Developments in Hodge Theory (Proceedings of Hodge Theory at IMSA)},
   date={2020},
}



\bib{KM}{book}{
   author={Koll{\'a}r, J.},
   author={Mori, S.},
   title={Birational geometry of algebraic varieties},
   series={Cambridge Tracts in Mathematics},
   volume={134},
   note={With the collaboration of C. H. Clemens and A. Corti;
   Translated from the 1998 Japanese original},
   publisher={Cambridge University Press, Cambridge},
   date={1998},
}

\bib{Kovacs1}{article}{
   author={Kov\'{a}cs, S.},
   title={Rational, log canonical, Du Bois singularities: on the conjectures
   of Koll\'{a}r and Steenbrink},
   journal={Compositio Math.},
   volume={118},
   date={1999},
   number={2},
   pages={123--133},
   issn={0010-437X},
}

\bib{KoS}{article}{
   author={Kov\'acs, S.},
   author={Schwede, K.},
   title={Du Bois singularities deform},
   journal={Minimal models and extremal rays (Kyoto, 2011), Adv. Stud. in Pure Math., Math. Soc. Japan},
   volume={70},
   date={2016},
   pages={49--65},
}

\bib{Matsumura}{book}{
   author={Matsumura, H.},
   title={Commutative ring theory},
   series={Cambridge Studies in Advanced Mathematics},
   volume={8},
   edition={2},
   note={Translated from the Japanese by M. Reid},
   publisher={Cambridge University Press, Cambridge},
   date={1989},
   pages={xiv+320},
}

\bib{MOPW}{article}{
   author={Musta\c{t}\u{a}, M.},
   author={Olano, S.},
   author={Popa, M.},
   author={Witaszek, J.},
   title={The Du Bois complex of a hypersurface and the minimal exponent},
   journal={Duke Math. J.},
   volume={172},
   date={2023},
   number={7},
   pages={1411--1436},
}

\bib{MP0}{article}{
      author={Musta\c{t}\u{a}, M.},
      author={Popa, M.},
      title={Hodge ideals},
      journal={Memoirs of the AMS}, 
      volume={262},
      date={2019}, 
      number={1268},
}

\bib{MP3}{article}{
   author={Musta\c{t}\u{a}, M.},
   author={Popa, M.},
   title={Hodge ideals for ${\bf Q}$-divisors, $V$-filtration, and minimal
   exponent},
   journal={Forum Math. Sigma},
   volume={8},
   date={2020},
   pages={Paper No. e19, 41pp},
}




\bib{MP1}{article}{
     author={Musta\c{t}\u{a}, M.},
     author={Popa, M.},
     title={Hodge filtration, minimal exponent, and local vanishing},
     journal={ Invent. Math.},
     volume={220},
     date={2020},
     pages={453--478},
}



\bib{MP2}{article}{
author={Musta\c{t}\u{a}, M.},
author={Popa, M.},
title={Hodge filtration on local cohomology, Du Bois complex, and local cohomological dimension},
journal={ Forum of Math. Pi},
volume={10}
date={2022},
pages={Paper No. e22, 58pp},
}

\bib{Namikawa}{article}{
   author={Namikawa, Y.},
   title={Deformation theory of singular symplectic $n$-folds},
   journal={Math. Ann.},
   volume={319},
   date={2001},
   pages={597--623},
}

\bib{PS}{book}{
   author={Peters, C.},
   author={Steenbrink, J.},
   title={Mixed Hodge structures},
   series={Ergebnisse der Mathematik und ihrer Grenzgebiete. 3. Folge. A
   Series of Modern Surveys in Mathematics [Results in Mathematics and
   Related Areas. 3rd Series. A Series of Modern Surveys in Mathematics]},
   volume={52},
   publisher={Springer-Verlag, Berlin},
   date={2008},
   pages={xiv+470},
}

\bib{Saito-MHP}{article}{
   author={Saito, M.},
   title={Modules de Hodge polarisables},
   journal={Publ. Res. Inst. Math. Sci.},
   volume={24},
   date={1988},
   pages={849--995},
}

\bib{Saito_duality}{article}{
   author={Saito, M.},
   title={Duality for vanishing cycle functors},
   journal={Publ. Res. Inst. Math. Sci.},
   volume={25},
   date={1989},
   pages={889--921},
}


\bib{Saito-MHM}{article}{
   author={Saito, M.},
   title={Mixed Hodge modules},
   journal={Publ. Res. Inst. Math. Sci.},
   volume={26},
   date={1990},
   pages={221--333},
}


\bib{Saito-B}{article}{
   author={Saito, M.},
   title={On $b$-function, spectrum and rational singularity},
   journal={Math. Ann.},
   volume={295},
   date={1993},
   number={1},
   pages={51--74},
}

\bib{Saito_microlocal}{article}{
   author={Saito, M.},
   title={On microlocal $b$-function},
   journal={Bull. Soc. Math. France},
   volume={122},
   date={1994},
   pages={163--184},
}

\bib{Saito-HC}{article}{
   author={Saito, M.},
   title={Mixed Hodge complexes on algebraic varieties},
   journal={Math. Ann.},
   volume={316},
   date={2000},
   pages={283--331},
}


\bib{Saito-LOG}{article}{
   author={Saito, M.},
   title={Direct image of logarithmic complexes and infinitesimal invariants of cycles},
   conference={
      title={Algebraic cycles and motives. Vol. 2},
   },
   book={
      series={London Math. Soc. Lecture Note Ser.},
      volume={344},
      publisher={Cambridge Univ. Press, Cambridge},
   },
   date={2007},
   pages={304--318},
   }


\bib{Saito-MLCT}{article}{
      author={Saito, M.},
	title={Hodge ideals and microlocal $V$-filtration},
	journal={preprint arXiv:1612.08667}, 
	date={2016}, 
}

\bib{Saito-AppendixFL}{article}{
      author={Saito, M.},
	title={Appendix to \cite{FL2}},
	date={2022}, 
}

\bib{Steenbrink2}{article}{
   author={Steenbrink, J. H. M.},
   title={Mixed Hodge structures associated with isolated singularities},
   conference={
      title={Singularities, Part 2},
      address={Arcata, Calif.},
      date={1981},
   },
   book={
      series={Proc. Sympos. Pure Math.},
      volume={40},
      publisher={Amer. Math. Soc., Providence, RI},
   },
   date={1983},
   pages={513--536},
}

\bib{Steenbrink}{article}{
   author={Steenbrink, J.},
   title={Vanishing theorems on singular spaces},
   note={Differential systems and singularities (Luminy, 1983)},
   journal={Ast\'{e}risque},
   number={130},
   date={1985},
   pages={330--341},
}


\end{biblist}
\end{document}